\documentclass[11pt, a4paper]{article}
\usepackage{mathrsfs, mathtools, amsthm, amssymb, thm-restate}  
\usepackage{paralist, tablists}  

\usepackage[left=25mm, top=25mm, bottom=25mm, right=25mm]{geometry}  
\usepackage[T1]{fontenc} 

\usepackage[inline]{enumitem}  
\usepackage[square, numbers, sort&compress]{natbib}  
\usepackage[
  bookmarks=true,                
  bookmarksnumbered=true,        
  bookmarksopen=true,            
  plainpages=false,              
  colorlinks=true,               
  linkcolor=blue,                
  citecolor=black,               
  anchorcolor=green,             
  urlcolor=blue,                 
  hyperindex=true                
]{hyperref} 

\newtheorem{theorem}{Theorem}[section] 
\newtheorem{lemma}[theorem]{Lemma}
\newtheorem{conjecture}{Conjecture}
\newtheorem{corollary}[theorem]{Corollary}

\newtheorem{claim}{Claim}
\theoremstyle{definition}

\newtheorem{remark}{Remark}

\makeatletter
\def\th@plain{%
  \upshape 
}
\makeatother

\newcommand{\etal}{et~al.\ }

\usepackage{marvosym,wasysym}
\usepackage{array}
\usepackage{bm}  
\DeclareMathOperator{\mad}{mad}
\renewcommand{\emph}{\textbf}

\newcounter{Hcase}
\newcounter{Hclaim}

\newcommand{\resetcounter}{\stepcounter{Hcase}\setcounter{case}{0}\stepcounter{Hclaim}\setcounter{claim}{0}}

\makeatletter
\renewenvironment{proof}[1][\proofname]{\par
  \pushQED{\qed}%
  \normalfont \topsep6\p@\@plus6\p@\relax
  \trivlist
  \item[\hskip\labelsep
        \bfseries
    #1\@addpunct{.}]\ignorespaces
}{%
  \popQED\endtrivlist\@endpefalse
}
\makeatother

\usepackage[capitalise]{cleveref}  
\crefname{claim}{Claim}{Claims}

\usepackage{harpoon}

\title{On odd colorings of sparse graphs}
\author{Tao Wang\thanks{Center for Applied Mathematics, Henan University, Kaifeng, P. R. China.
\texttt{wangtao@henu.edu.cn}
}
\and Xiaojing Yang\thanks{School of Mathematics and Statistics, Henan University, Kaifeng, P. R. China.
\texttt{yangxiaojing@henu.edu.cn}
}}

\begin{document}
\date{November 13, 2023}
\maketitle

\begin{abstract}
An \emph{odd $c$-coloring} of a graph is a proper $c$-coloring such that each non-isolated vertex has a color appearing an odd number of times within its open neighborhood. A \emph{proper conflict-free $c$-coloring} of a graph is a proper $c$-coloring such that each non-isolated vertex has a color appearing exactly once within its neighborhood. Clearly, every proper conflict-free $c$-coloring is also an odd $c$-coloring. Cranston conjectured that every graph $G$ with maximum average degree $\mad(G) < \frac{4c}{c+2}$ (where $c \geq 4$) has an odd $c$-coloring, and he proved this conjecture for $c\in\{5, 6\}$. Note that the bound $\frac{4c}{c+2}$ is best possible. Cho \etal solved Cranston's conjecture for $c \geq 5$, strengthening the result by transitioning from odd $c$-coloring to proper conflict-free $c$-coloring. However, they did not provide all the extremal non-colorable graphs $G$ with $\mad(G) = \frac{4c}{c+2}$, which remains an open question of interest. 

In this paper, we tackle this intriguing extremal problem. We aim to characterize all non-proper conflict-free $c$-colorable graphs $G$ with $\mad(G) = \frac{4c}{c+2}$. For the case of $c=4$, Cranston's conjecture is not true, as evidenced by the existence of a counterexample: a graph whose every block is a $5$-cycle. Cho \etal proved that a graph $G$ with $\mad(G) < \frac{22}{9}$ and no induced $5$-cycles has an odd $4$-coloring. We improve this result by proving that a graph $G$ with $\mad(G) \leq \frac{22}{9}$ (with equality allowed) is not odd $4$-colorable if and only if $G$ belongs to a specific class of graphs. On the other hand, Cho \etal established that a planar graph with girth at least $5$ has an odd $6$-coloring; we improve it by proving that a planar graph without $4^{-}$-cycles adjacent to $7^{-}$-cycles also has an odd $6$-coloring.



\end{abstract}

\section{Introduction}

All graphs considered in this paper are finite and simple. An \emph{odd $c$-coloring} of a graph is a proper $c$-coloring with the additional constraint that each non-isolated vertex must have a color appearing an odd number of times within its open neighborhood. A graph is \emph{odd $c$-colorable} if it has an odd $c$-coloring. The \emph{odd chromatic number} of a graph $G$, denoted by $\chi_{o}(G)$, is the smallest integer $c$ such that $G$ admits an odd $c$-coloring.

Petru\v{s}evski and \v{S}krekovski~\cite{MR4467654} introduced the concept of odd coloring, which can be seen as a relaxation of proper conflict-free coloring (as defined later). Petru\v{s}evski and \v{S}krekovski~\cite{MR4467654} proved that planar graphs are odd $9$-colorable, and proposed the following conjecture: 
\begin{conjecture}
Every planar graph is odd 5-colorable.
\end{conjecture}

Note that the bound $5$ is best possible if the conjecture is true, as there exists a class of graphs $\mathcal{H}$ (defined later) which are planar graphs with odd chromatic number $5$. Petr and Portier~\cite{MR4559435} proved that planar graphs are odd 8-colorable, and Metrebian~\cite{Metrebian2022} proved that toroidal graphs are odd $9$-colorable. Cranston \etal \cite{MR4537616} proved that every $1$-planar graph is odd $23$-colorable, where a graph is \emph{$1$-planar} if it can be drawn on the plane such that every edge is crossed by at most one other edge. Niu and Zhang~\cite{Niu2022} improved this result by showing that every $1$-planar graph is odd $16$-colorable. Recently, Liu \etal \cite{MR4564916} improved the result to odd $13$-colorability. For odd coloring of $k$-planar graphs, we refer the reader to~\cite{Dujmovic2022,MR4654344}.

Cranston~\cite{Cranston2022} investigated odd colorings of sparse graphs, measured in terms of the maximum average degree. Given a graph $G$, the \emph{maximum average degree} of $G$, denoted by $\mad(G)$, is the maximum value of $2|E(H)|/|V(H)|$, taken over all non-empty subgraphs $H$ of $G$. That is defined as:
\[
\mad(G) = \max_{\emptyset \neq H \subseteq G} \left\{\frac{2|E(H)|}{|V(H)|}\right\}.
\]

For $c\geq 4$, Cranston~\cite{Cranston2022} proposed the following conjecture:

\begin{conjecture}[Cranston~\cite{Cranston2022}]\label{Conj:MAD}
For $c \geq 4$, if $G$ is a graph with $\mad(G) < \frac{4c}{c+2}$, then $\chi_{o}(G)\leq c$.
\end{conjecture}

Let $\mathrm{SK}_{n}$ be the graph obtained from the complete graph on $n$ vertices by subdividing every edge precisely once. Note that $\mad(\mathrm{SK}_{c+1})=\frac{4c}{c+2}$ and $\chi_{o}(\mathrm{SK}_{c+1})=c+1$. Before introducing the conjecture, Cranston proved that a graph $G$ with $\mad(G) < \frac{4c}{c+2}$ is odd $(c+3)$-colorable. To support \autoref{Conj:MAD}, Cranston also verified the conjecture for $c \in \{5, 6\}$ in a stronger form.

Cho \etal \cite{MR4533825} proved the conjecture in the following form.

\begin{theorem}[Cho \etal \cite{MR4533825}]
For $c \geq 7$, if $G$ is a graph with $\mad(G) \leq \mad(\mathrm{SK}_{c+1}) = \frac{4c}{c+2}$, then $\chi_{o}(G) \leq c$ unless $G$ contains $\mathrm{SK}_{c+1}$ as a subgraph.
\end{theorem}

Actually, some graphs that contain $\mathrm{SK}_{c+1}$ as a subgraph are still odd $c$-colorable. In other words, Cho \etal \cite{MR4533825} did not provide a complete characterization of non-odd $c$-colorable graphs with $\mad(G) = \frac{4c}{c+2}$.

Later, Cho \etal \cite{Cho2022a} proved \autoref{Conj:MAD} in terms of proper conflict-free coloring. A \emph{proper conflict-free $c$-coloring} (abbreviated as PCF $c$-coloring) of a graph is a proper $c$-coloring such that each non-isolated vertex has a color appearing precisely once within its neighborhood. A graph is \emph{PCF $c$-colorable} if it has a PCF $c$-coloring. The \emph{PCF chromatic number} of a graph $G$, denoted by $\chi_{\mathrm{pcf}}(G)$, is the minimum integer $c$ such that $G$ admits a PCF $c$-coloring. Since every PCF $c$-coloring is also an odd $c$-coloring, it is straightforward that $\chi_{o}(G) \leq \chi_{\mathrm{pcf}}(G)$. Fabrici \etal \cite{MR4493821} proved that every planar graph is PCF $8$-colorable. Caro \etal \cite{MR4499342} gave some results weaker than the following \autoref{PCF-MAD-Cho}.

\begin{theorem}[Cho \etal \cite{Cho2022a}]\label{PCF-MAD-Cho}
For $c \geq 5$, if $G$ is a graph with $\mad(G) \leq \mad(\mathrm{SK}_{c+1}) = \frac{4c}{c+2}$, then $\chi_{\mathrm{pcf}}(G) \leq c$ unless $G$ contains $\mathrm{SK}_{c+1}$ as a subgraph.
\end{theorem}

Similarly, Cho \etal \cite{Cho2022a} did not provide a comprehensive characterization of graphs that are non-proper conflict-free $c$-colorable graphs when $\mad(G) = \frac{4c}{c+2}$.
In this paper, we provide a complete resolution of \autoref{Conj:MAD} with $c \geq 5$, focusing on proper conflict-free coloring. We also establish a characterization of all the extremal cases.

Let $\mathcal{G}_{c}$ be the class of graphs $H$ that contain an induced subgraph $H'$ isomorphic to $\mathrm{SK}_{c+1}$, where each $2$-vertex in $H'$ also serves as a $2$-vertex in $H$. For convenience, we refer to $H'$ as a \emph{bad structure} of $H$.

\begin{theorem}\label{PCF-MAD-WY}
For an integer $c \geq 5$, if $G$ is a graph with $\mad(G) \leq \mad(\mathrm{SK}_{c+1}) = \frac{4c}{c+2}$, then $\chi_{\mathrm{pcf}}(G) \geq c+1$ if and only if $G \in \mathcal{G}_{c}$.
\end{theorem}

For the case when $c = 4$, \autoref{Conj:MAD} is not valid, as $\mad(C_5) = 2 < \frac{8}{3}$ yet $\chi_{o}(C_5)=5>4$. Let $\mathcal{H}_{t}$ be the class of graphs that contain a component consisting of $t$ blocks where $t \geq 1$, with each block being a cycle of length $5$ (referred to as a $5$-cycle). We prove that every graph in $\mathcal{H}_{t}$ is not odd $4$-colorable using an induction approach on $t$. In $\mathcal{H}_{1}$, every graph  has a component isomorphic to a $5$-cycle, which is not odd $4$-colorable. Assume $G$ is a graph in $\mathcal{H}_{t}$ with $t \geq 2$, and $K$ is a component consisting of $t$ blocks, with each block being a $5$-cycle. Let $B = [uvwxy]$ be an end-block of $K$, where $v$ is the cut vertex. If $G$ has an odd $4$-coloring $\phi$, then $u$ and $w$ must be colored the same. Then the restriction of $\phi$ on $G - \{u, w, x, y\}$ still forms an odd $4$-coloring, but $G - \{u, w, x, y\}$ belongs to $\mathcal{H}_{t-1}$. Hence, every graph in $\mathcal{H}_{t}$ is not odd $4$-colorable. Therefore, every graph in $\mathcal{H}_{t}$ is not PCF $4$-colorable. Caro et al.~\cite{MR4499342} established that every graph in $\mathcal{H}_{t}$ has PCF chromatic number $5$.

Since \autoref{Conj:MAD} is disproven for $c=4$, Cho \etal \cite{MR4533825} proved the following result.

\begin{theorem}[Cho \etal \cite{MR4533825}]
If $G$ is a graph with $\mad(G) < \frac{22}{9}$ and contains no induced $5$-cycles, then $\chi_{o}(G) \leq 4$.
\end{theorem}

However, it is essential to note that a substantial class of graphs containing induced $5$-cycles can still be odd $4$-colorable. Therefore, we investigate a class of graphs with bounded maximum average degree, excluding $\mathcal{H}_{t}$ as exceptions, and obtain the following result. Let $\mathcal{H}$ be the class of graphs, each of which is in $\mathcal{H}_{t}$ with $t \geq 1$.

\begin{theorem}\label{Thm:Odd4colorable}
Let $G$ be a graph with $\mad(G) \leq \frac{22}{9}$. Then $\chi_{o}(G) \geq 5$ if and only if $G \in \mathcal{H}$.
\end{theorem}

It is worth mentioning that the upper bound on the maximum average degree, $\frac{22}{9}$, is less than the conjectured $\frac{8}{3}$ by Cranston. Consequently, it is interesting to consider whether the bound $\frac{22}{9}$ can be replaced with $\frac{8}{3}$ in \autoref{Thm:Odd4colorable}.

Results regarding bounded maximum average degree have natural corollaries to planar graphs with specific girth (minimum cycle length) restrictions.
Namely, considering graphs $G$ with $\mad(G)<\frac{2g}{g-2}$ is relevant because it includes all planar graphs with girth at least $g$. Note that for each integer $c \geq 5$, the graph $\mathrm{SK}_{c+1}$ is non-planar. Hence, by \autoref{PCF-MAD-Cho} or \autoref{PCF-MAD-WY}, we have the following corollary:

\begin{corollary}[Cho \etal \cite{Cho2022a}]\label{Cor:Planar-PCF}
For $c \geq 5$, every planar graph with girth at least $\left\lceil \frac{4c}{c-2} \right\rceil$ is PCF $c$-colorable.
\end{corollary}

\autoref{Cor:Planar-PCF} confirms that every planar graph with girth at least $6$ is PCF $6$-colorable, and every planar graph with girth at least $7$ is PCF $5$-colorable. Cho \etal \cite{Cho2022a} decreased the girth restriction by proving that every planar graph with girth at least $5$ is PCF $7$-colorable. For odd coloring, Cho \etal \cite{MR4533825} proved that every planar graph with girth at least $11$ is odd $4$-colorable. Moreover, they proved the following result.

\begin{theorem}[Cho \etal \cite{MR4533825}]
Every planar graph with girth at least $5$ is odd $6$-colorable.
\end{theorem}

Building upon this result, we give an enhancement as follows:

\begin{theorem}\label{Thm:Odd6colorable}
Let $G$ be a planar graph without $4^{-}$-cycles adjacent to $7^{-}$-cycles. Then $G$ is odd $6$-colorable.
\end{theorem}

As we conclude this section, let us introduce some essential notations and terminologies. Most of them follow that in~\cite{MR4533825}. Given a graph $G$ and $S \subseteq V(G)$, let $G-S$ denote the subgraph of $G$ induced by $V(G) \setminus S$. For a vertex $v$, let $N_G(v)$ denote the neighborhood of $v$ and $d_G(v)$ denote the degree of $v$. A \emph{$d$-vertex} (resp. \emph{$d^{+}$-vertex} and \emph{$d^{-}$-vertex}) is a vertex of degree exactly $d$ (resp. at least $d$ and at most $d$). Similarly, a \emph{$d$-neighbor} of a vertex $v$ is a $d$-vertex that is a neighbor of $v$. Let $N_{G,d}(v)$ be the set of all $d$-neighbors of $v$, and $n_{G,d}(v)=|N_{G,d}(v)|$. Similarly, $d^{+}$-neighbor, $d^{-}$-neighbor, $N_{G, d^{+}}(v)$, $N_{G, d^{-}}(v)$, $n_{G, d^{+}}(v)$, and $n_{G, d^{-}}(v)$ are defined analogously. If there is no confusion, then we often omit the subscript $G$ in these notations, as in $d(v)$, $N(v)$, $N_d(v)$, $n_{d^{+}}(v)$, and so on. An \emph{$\ell$-thread} represents a path with at least $\ell$ vertices, each having degree $2$ in the graph.

Let $\phi$ be a partial proper coloring of a graph $G$. A \emph{PCF color} of a vertex $v$ is a color that appears exactly once in the neighborhood $N_G(v)$. If such a color exists, $\phi_{*}(v)$ denotes that color; otherwise, $\phi_{*}(v)$ remains undefined. In the course of extending a partial coloring $\phi$ to the entire graph, we will abuse notation and use $\phi_{*}(v)$ to denote a PCF color of $v$ under the \emph{current coloring} within the ongoing coloring process. For two vertices $x$ and $y$, when we say ``color $x$ with a color that is not $\phi(y)$ (or $\phi_{*}(y)$)'', we exclude the color $\phi(y)$ (or $\phi_{*}(y)$) only if that color is currently defined. In addition, for technical reasons, we allow a PCF $c$-coloring of the null graph, which is a graph with empty vertex set.

\begin{remark}\label{rem1}
For a partial coloring $\phi$ of $G$, suppose that a vertex $v$ has precisely two colored neighbors, namely $u$ and $w$. Then $|\phi(N(v))| = 2$ or $|\phi(N(v))| = 1$. In the former case, both $\phi(u)$ and $\phi(w)$ appear exactly once in $N(v)$, we define $\phi_{*}(v)$ as either $\phi(u)$ or $\phi(w)$. In the latter case, there is only one color in $N(v)$. If there exists an uncolored neighbor $x$ of $v$, then we can assign $x$ a color different from $\phi(u)$ to ensure that $\phi_{*}(v)$ is defined once $x$ receives its color. In other words, to ensure that $v$ has a PCF color, we only need to avoid using $\phi(u)$ on $x$ in the later case. In summary, if a vertex $v$ has precisely two colored neighbors, we can let $\phi_{*}(v)$ be defined as any color in $N(v)$. This will be used in the proof of \autoref{Lem:semi-PCF} and \autoref{Lem:semi-PCF-2}.
\end{remark}

Analogously, an \emph{odd color} of a vertex is one that appears an odd number of times within its neighborhood. We write $\phi_{o}(v)$ to denote the unique color which appears an odd number of times in its neighborhood if such a unique color exists.

\section{Proper conflict-free coloring of sparse graphs}
In this section, we focus on proper conflict-free coloring of sparse graphs. We introduce essential definitions as follows. 

Let $Y$ be a subset of $V(G)$, and let $\phi$ be a proper $c$-coloring of $G - Y$. We say that $\phi$ is a \emph{semi-PCF $c$-coloring} of $(G, Y)$ if every vertex in $G - Y - (N_{G}(Y) \cap Z)$ has a PCF color with respect to $\phi$, where $Z$ is the set of vertices in $G-Y$ having degree precisely two in $G-Y$.

\begin{lemma}\label{Lem:semi-PCF}
Let $c \geq 5$ be an integer, and $v$ be a vertex with a $2$-neighbor $u_{0}$. If $G$ has no PCF $c$-colorings, but $(G, \{v\} \cup N_{G,1}(v) \cup N_{G,2}(v))$ has a semi-PCF $c$-coloring, then 
\[
2d(v) - 2n_{1}(v) - n_{2}(v) \geq c.
\]
\end{lemma}

\begin{proof}
Let $X = N_{G}(N_{G,2}(v)) \setminus (\{v\} \cup N_{G,2}(v))$, $Y = \{v\} \cup N_{G,1}(v) \cup N_{G,2}(v)$, and $Z = \{u \in G-Y : d_{G-Y}(u) = 2\}$. Suppose to the contrary that $2d(v) - 2n_{1}(v) - n_{2}(v) \leq c-1$. By the assumption, we know that $(G, Y)$ has a semi-PCF $c$-coloring $\phi$. Note that every vertex $z$ in $N(Y) \cap Z$ has degree exactly two in $G - Y$, thus it has two distinct colors in $N_{G-Y}(z)$ or only a single color in $N_{G-Y}(z)$. Firstly, we assign a color not in $C= \phi(X\cup N_{G,3^{+}}(v)) \cup \phi_{*}(N_{G,3^{+}}(v))$ to the vertex $v$.
Note that
\[
\begin{array}{rcl}
|C| \leq |X \cup N_{G,3^{+}}(v)| + |N_{G,3^{+}}(v)| & \leq &
(d(v) - n_{1}(v)) + (d(v) - n_{1}(v) - n_{2}(v))\\ &=& 2d(v) - 2n_{1}(v) - n_{2}(v)\\ & \leq & c-1,
\end{array}
\]
so there is at least one color to use on $v$. Next, proceed with coloring the vertices in $N_{G,1}(v) \cup N_{G,2}(v) = \{u_{0}, u_{1}, \dots\}$ one by one. Note that every vertex in $N_{G, 2}(v)$ with a neighbor in $G - Y$ already has a PCF color under the current coloring. If $v$ has a PCF color, then we assign $u_{0}$ a color not in $\phi(N(u_{0})) \cup \phi_{*}(N(u_{0}))$, otherwise we assign $u_{0}$ a color not in $C_{0} = \phi(N(u_{0})) \cup \phi_{*}(N(u_{0})) \cup \phi(N_{G, 3^{+}}(v))$. Note that this is feasible since
\[
|C_{0}| \leq 3 + \left\lfloor \frac{n_{3^{+}}(v)}{2} \right\rfloor \leq 3 + \left\lfloor\frac{c-1-n_{2}(v)}{4}\right\rfloor \leq c-1. \text{ (Note that $n_{2}(v) \geq 1$)}
\]
After $u_{0}$ is colored, $\phi_{*}(v)$ is defined since $\phi_{*}(v) = \phi(u_{0})$. For other $2^{-}$-vertex $u_{i}$ in $\{u_{1}, u_{2}, \dots\}$, we assign a color not in $\phi(N_{G}(u_{i}))\cup \phi_{*}(N_{G}(u_{i}))$ to $u_{i}$. Note that $|\phi(N_{G}(u_{i}))\cup \phi_{*}(N_{G}(u_{i}))|\leq 4\leq c-1$. At this point, we have a PCF $c$-coloring of $G$, a contradiction. This completes the proof of \autoref{Lem:semi-PCF}. 
\end{proof}

Similarly, we can prove the following result for $c \geq 7$.
\begin{lemma}\label{Lem:semi-PCF-2}
Let $c \geq 7$ be an integer, and let $v$ be a vertex having a $2$- or $3$-neighbor $u_{0}$. If $G$ has no PCF $c$-coloring, but $(G, \{v\} \cup N_{G,1}(v) \cup N_{G,2}(v) \cup N_{G,3}(v))$ has a semi-PCF $c$-coloring, then $2d(v) - 2n_{1}(v) - n_{2}(v) - n_{3}(v) \geq c$.
\end{lemma}
\begin{proof}
Let $X$ be a set of vertices such that $|X \cap (N_{G}(x)\setminus\{v\})| = 1$ for each $x \in N_{G,2}(v) \cup N_{G,3}(v)$, $Y = \{v\} \cup N_{G,1}(v) \cup N_{G,2}(v) \cup N_{G,3}(v)$, and $Z = \{u \in G-Y : d_{G-Y}(u) = 2\}$. Suppose to the contrary that $2d(v) - 2n_{1}(v) - n_{2}(v) -n_{3}(v) \leq c-1$.
By the assumption, $(G, Y)$ has a semi-PCF $c$-coloring $\phi$. Note that every vertex $z$ in $N(Y) \cap Z$ has degree exactly two in $G - Y$, thus it has two distinct colors in $N_{G-Y}(z)$ or only a single color in $N_{G-Y}(z)$. First, we assign a color not in $C= \phi(X\cup N_{G,4^{+}}(v)) \cup \phi_{*}(N_{G,4^{+}}(v))$ to the vertex $v$.
Note that
\[
\begin{array}{rcl}
|C| \leq |X \cup N_{G,4^{+}}(v)| + |N_{G,4^{+}}(v)| & \leq &
(d(v) - n_{1}(v)) + (d(v) - n_{1}(v) - n_{2}(v) - n_{3}(v))\\ &=& 2d(v) - 2n_{1}(v) - n_{2}(v) - n_{3}(v)\\ & \leq & c-1,
\end{array}
\]
so there is at least one color to use on $v$. Next, start coloring the vertices in $N_{G,1}(v) \cup N_{G,2}(v) \cup N_{G,3}(v) = \{u_{0}, u_{1}, \dots\}$ one by one. Note that every vertex in $N_{G, 2}(v)$ with a neighbor in $G - Y$ already has a PCF color under the current coloring. If $v$ has a PCF color, then we assign $u_{0}$ a color not in $\phi(N(u_{0})) \cup \phi_{*}(N(u_{0}))$, otherwise we assign $u_{0}$ a color not in $C_{0} = \phi(N(u_{0})) \cup \phi_{*}(N(u_{0})) \cup \phi(N_{G, 4^{+}}(v))$. Note that this is feasible since
\[
|C_{0}| \leq 5 + \left\lfloor \frac{n_{4^{+}}(v)}{2} \right\rfloor \leq 5 + \left\lfloor\frac{c-1-n_{2}(v)-n_{3}(v)}{4}\right\rfloor \leq c-1. \text{ (Note that $n_{2}(v) + n_{3}(v) \geq 1$)}
\]
After $u_{0}$ is colored, $\phi_{*}(v) = \phi(u_{0})$ is defined. For other $3^{-}$-vertex $u_{i}$ in $\{u_{1}, u_{2}, \dots\}$, we assign a color not in $\phi(N_{G}(u_{i}))\cup \phi_{*}(N_{G}(u_{i}))$ to $u_{i}$. Note that $|\phi(N_{G}(u_{i}))\cup \phi_{*}(N_{G}(u_{i}))|\leq 6 \leq c-1$. At this point, we have a PCF $c$-coloring of $G$, a contradiction. This completes the proof. 
\end{proof}

\noindent\textbf{Proof of \autoref{PCF-MAD-WY}.} We observe that every graph in $\mathcal{G}_{c}$ has PCF chromatic number at least $c + 1$, so it suffices to prove the reverse implication in \autoref{PCF-MAD-WY}. Fix an integer $c \geq 5$, and let $G$ be a counterexample to \autoref{PCF-MAD-WY} with the minimum number of vertices. Then $G$ is a graph with $\mad(G)\leq \frac{4c}{c+2}$ and $G \notin \mathcal{G}_{c}$ where $G$ has no PCF $c$-coloring.

\begin{claim}\label{Claim:semi-pcf}
For any nonempty proper subset $Y \subset V(G)$, we always have that $(G, Y)$ has a semi-PCF $c$-coloring.
\end{claim}
\begin{proof}
Consider $G'= G - Y$. If $G'$ has a PCF $c$-coloring, then we are done. Otherwise, assume $G'$ has no PCF $c$-coloring. By the minimality, $G'$ is isomorphic to a graph in $\mathcal{G}_{c}$, where $G'$ has an induced subgraph $H$ isomorphic to $\mathrm{SK}_{c+1}$, and each $2$-vertex of $H$ is also a $2$-vertex of $G'$. Recall that $H$ is called a ``bad structure'' of $G'$. Note that $G'$ may have more than one bad structure. Let $\mathcal{B}$ be the class of bad structures of $G'$. Since $\mad(G) \leq \mad(\mathrm{SK}_{c+1})$ and the average degree of a bad structure is the same as $\mad(\mathrm{SK}_{c+1})$, it follows that any two distinct bad structures in $\mathcal{B}$ are disjoint. For every $H$ in $\mathcal{B}$, there exists a $2$-vertex $u_{H}$ of $H$ with a neighbor in $Y$, as otherwise, $H$ would be a bad structure of $G$. It is obvious that each neighbor of $u_{H}$ in $H$ has degree at least $c$ in $H$ (and hence in $G'$). For each $H$ in $\mathcal{B}$, we only need to associate it with one such $2$-vertex $u_{H}$ even though there may be many such $2$-vertices. Let $U = \{u_{H} : H \in \mathcal{B}\}$. Note that the deletion of $U$ from $G'$ destroys all the $H$'s in $\mathcal{B}$. Recall that distinct bad structures in $\mathcal{B}$ are disjoint and, for every $u \in U$, every neighbor of $u$ has degree at least $c$ in $G'$. This implies that every neighbor of $u$ has degree at least $c - 1 \geq 4$ in $G' - U$. Therefore, the deletion of $U$ from $G'$ does not introduce new bad structures. Hence, $G' - U$ is not in $\mathcal{G}_{c}$, and it must have a PCF $c$-coloring $\phi$. Note that $U$ is an independent set. For each vertex $u$ in $U$, we assign a color not in $\phi(N_{G'}(u)) \cup \phi_{*}(N_{G'}(u))$ to $u$. Since $|\phi(N_{G'}(u)) \cup \phi_{*}(N_{G'}(u))| \leq 4 \leq c - 1$, there is at least one admissible color for each $u$. Thus, the resulting coloring is a semi-PCF $c$-coloring of $(G, Y)$.
\end{proof}

Now, let us sketch how to extend a semi-PCF coloring obtained from \autoref{Claim:semi-pcf} to the entire graph $G$. By the proof of \autoref{Claim:semi-pcf}, every vertex that has no PCF color is either an isolated vertex in $G - Y$ or in $U$. For each isolated vertex $u$ in $G - Y$, we simply color one of its neighbors $w$ properly to ensure that $u$ has a PCF color in the resulting coloring. For every $2$-vertex $u$ in $U$, it has at least one neighbor in $Y$. When we assign a color to a neighbor $w$ of $u$ in $Y$, we only need to avoid the colors already assigned to $u$ and one of its neighbors in $G'$. Since $u$ has at most two colored neighbors (as described in \autoref{rem1}), we can simply ensure that $u$ has a PCF color by coloring $w$ accordingly. In order to count the number of forbidden colors, we also use $\phi_{*}(v)$ to denote a color on a neighbor when $v$ has at most two colored neighbors as in \autoref{rem1}. By following these steps, we extend the semi-PCF coloring to the entire graph $G$, ensuring that each vertex has a PCF color. This contradicts our assumption that $G$ does not have a PCF $c$-coloring, thus completing the proof.

\begin{claim}\label{Claim:pcf-n1}
$G$ has no $1$-vertex.
\end{claim}
\begin{proof}
Suppose to the contrary that $x$ is a $1$-vertex in $G$, and $u$ is the unique neighbor of $x$ in $G$. By \autoref{Claim:semi-pcf}, $(G, x)$ has a semi-PCF $c$-coloring $\phi$. Then every vertex other than $u$ has a PCF color. If $u$ has a PCF color, then we use a color not in $\{\phi(u), \phi_{*}(u)\}$ on $x$ to obtain a PCF $c$-coloring of $G$, a contradiction. If $u$ has no PCF color, then $u$ is a $2$-vertex in $G-x$, and we assign a color to $x$ such that the resulting coloring is a PCF $c$-coloring of $G$, a contradiction.
\end{proof}

\begin{claim}\label{Claim:pcf-2thread}
No two $2$-vertices in $G$ are adjacent. 
\end{claim}
\begin{proof}
Suppose to the contrary that $x$ and $y$ are two adjacent $2$-vertices in $G$. By \autoref{Claim:semi-pcf}, $(G, \{x\} \cup N_{G,2}(x))$ has a semi-PCF $c$-coloring. According to \autoref{Lem:semi-PCF}, we deduce that $n_{2}(x) + c \leq 2d(x) = 4$, which is a contradiction.
\end{proof}

\begin{claim}\label{Claim:pcf-triangle}
No $2$-vertex is contained in a triangle of $G$.
\end{claim}
\begin{proof}
Suppose to the contrary that $xyz$ is a triangle with $d(x) = 2$. By \autoref{Claim:semi-pcf}, $(G, \{x\})$ has a semi-PCF $c$-coloring $\phi$. To obtain a PCF $c$-coloring of $G$, we assign a color not in the set $\{\phi(y), \phi_{*}(y), \phi(z), \phi_{*}(z)\}$ to $x$.
\end{proof}

\begin{claim}\label{Claim:pcf-easy3}
If a $3$-vertex $v$ has a $2$-neighbor, then $c = 5$ and $n_{2}(v) = 1$.
\end{claim}
\begin{proof}
By \autoref{Claim:semi-pcf}, $(G, \{v\} \cup N_{G,2}(v))$ has a semi-PCF $c$-coloring. It follows from \autoref{Lem:semi-PCF} that $n_{2}(v) + c \leq 6$. Then we deduce that $c = 5$ and $n_{2}(v) = 1$.
\end{proof}

\begin{claim}\label{Claim:combine}
If $v_{1}$ and $v_{2}$ are adjacent $3^{+}$-vertices in $G$, satisfying the condition: 
\begin{equation}\label{eq-1}
2d(v_{i}) - n_{2}(v_{i}) - 2 \leq c - i
\end{equation}
for each $i \in \{1, 2\}$, then $v_{1}$ is a $4^{+}$-vertex, and either $v_{1}$ or $v_{2}$ has no $2$-neighbors. Consequently,
\begin{enumerate}[label = (\roman*)]
\item if a $3$-vertex $v$ has a $2$-neighbor, then $c = 5$ and $v$ has no $3$-neighbors; and
\item if $c \geq 6$, then there are no adjacent $3$-vertices.
\end{enumerate}
\end{claim}
\begin{proof}
Let $Y = \{v_{1}, v_{2}\} \cup N_{2}(v_{1}) \cup N_{2}(v_{2})$. By \autoref{Claim:semi-pcf}, $(G, Y)$ has a semi-PCF $c$-coloring $\phi$. Let $C_{i} = \phi((N(N_{2}(v_{i})) \cup N_{3^{+}}(v_{i}))\setminus \{v_{1}, v_{2}\}) \cup \phi_{*}(N_{3^{+}}(v_{i}) \setminus \{v_{1}, v_{2}\})$ for each $i \in \{1, 2\}$. Note that $|C_{i}| \leq (d(v_{i}) - 1) + (d(v_{i}) - n_{2}(v_{i}) - 1) = 2d(v_{i}) - n_{2}(v_{i}) -2 \leq c - i$ for each $i \in \{1, 2\}$.

Assume that for each $i \in \{1, 2\}$, the $3^{+}$-vertex $v_{i}$ has a $2$-neighbor $u_{i}$, whose other neighbor is $x_{i}$. We assign a color not in $C_{1}$ to the vertex $v_{1}$, and a color not in $\{\phi(v_{1})\} \cup C_{2}$ to the vertex $v_{2}$. By \autoref{Claim:pcf-triangle}, $u_{1} \neq u_{2}$. Similar to the proof in \autoref{Lem:semi-PCF}, we can first color $u_{1}$ and $u_{2}$ to ensure that $\phi_{*}(v_{1})$ and $\phi_{*}(v_{2})$ are defined in the resulting coloring. Finally, we color the other $2$-vertices in $N_{2}(v_{1}) \cup N_{2}(v_{2})$ to obtain a PCF $c$-coloring of $G$, leading to a contradiction. Hence, either $v_{1}$ or $v_{2}$ has no $2$-neighbors.

Assume that $v_{1}$ is a $3$-vertex with neighbors $u_{1}, u_{2}$ and $v_{2}$. We consider three cases according to whether $\phi_{*}(v_{2})$ and $\phi_{*}(v_{1})$ are defined or not at this time. 

\textbf{Case 1.} $\phi_{*}(v_{2})$ is not defined.

In this case, we assign a color not in $C_{1}^{*} = C_{1} \cup \phi(N_{3^{+}}(v_{2})\setminus \{v_{1}\})$ to $v_{1}$. Note that
\[
\begin{array}{rcl}
|C_{1}^{*}| \leq |C_{1}| + |\phi(N_{3^{+}}(v_{2})\setminus \{v_{1}\})| \leq |C_{1}| + \left\lfloor\frac{d(v_{2}) - 1 - n_{2}(v_{2})}{2}\right\rfloor &\leq& |C_{1}| + \left\lfloor\frac{c - 2 - n_{2}(v_{2})}{4}\right\rfloor\\ [2mm]&\leq& |C_{1}| + c - 5 \leq c - 1,
\end{array}
\]
where the second inequality holds since $\phi_{*}(v_{2})$ is not defined, the third inequality holds since $2d(v_{2}) \leq n_{2}(v_{2}) + c$, the fourth inequality holds since $c \geq 5$, and the fifth inequality holds since $d(v_{1}) = 3$ and $|C_{1}| \leq 4$.
It is obvious that $\phi_{*}(v_{2})$ is defined in the resulting coloring. If $|\phi(N(v_{1}))| = 2$, then $\phi_{*}(v_{1})$ is defined, and we assign a color not in $C_{2} \cup \{\phi(v_{1})\}$ to $v_{2}$. If $|\phi(N(v_{1}))| \leq 1$, then $|C_{1}| \leq 3$ and $|C_{1}^{*}| \leq |C_{1}| + c - 5 \leq c-2$, we assign a color not in $C_{2} \cup \phi(N(v_{1}))$ to $v_{2}$, and (re)color $v_{1}$ with a color not in $C_{1}^{*} \cup \{\phi(v_{2})\}$. As a result, $\phi(v_{1}), \phi(v_{2}), \phi_{*}(v_{1})$ and $\phi_{*}(v_{2})$ are all defined.

\textbf{Case 2.} $\phi_{*}(v_{2})$ is defined, but $\phi_{*}(v_{1})$ is not defined.

Since $|N(v_{1}) \setminus \{v_{2}\}| = |\{u_{1}, u_{2}\}| = 2$ and $\phi_{*}(v_{1})$ is not defined, we have that either both $u_{1}$ and $u_{2}$ are uncolored, or $\phi(u_{1}) = \phi(u_{2})$. If $u_{1}$ and $u_{2}$ are uncolored, then both $u_{1}$ and $u_{2}$ are $2$-vertices, contradicting \autoref{Claim:pcf-easy3}. Therefore, we conclude that $\phi(u_{1}) = \phi(u_{2})$.
If $v_{2}$ has at least two PCF colors, then we first assign a color not in $C_{2} \cup \{\phi(u_{1})\}$ to $v_{2}$, and then assign a color not in $C_{1} \cup \{\phi(v_{2})\}$ to $v_{1}$. Note that $|C_{2} \cup \{\phi(u_{1})\}| \leq (c-2) + 1 = c-1$, and $|C_{1} \cup \{\phi(v_{2})\}| \leq 3 + 1 \leq c - 1$. If $v_{2}$ has exactly one PCF color, then we assign a color not in $C_{1} \cup \{\phi_{*}(v_{2})\}$ to $v_{1}$, and we assign a color not in $C_{2}^{*} = C_{2} \cup \{\phi(v_{1}), \phi(u_{1})\}$ to $v_{2}$. Note that $|C_{1} \cup \{\phi_{*}(v_{2})\}| \leq 3 + 1 \leq c-1$ and 
\[
|C_{2}^{*}| \leq 2 + |C_{2}| \leq 2 + \left(1 + \left\lfloor\frac{d(v_{2})-n_{2}(v_{2})-2}{2}\right\rfloor\right) \leq 3 + \left\lfloor\frac{c-4-n_{2}(v_{2})}{4}\right\rfloor \leq 3 + (c - 4) = c-1.
\]
In both subcases, $\phi(v_{1}), \phi(v_{2}), \phi_{*}(v_{1})$ and $\phi_{*}(v_{2})$ are all defined in the resulting coloring.

\textbf{Case 3.} Both $\phi_{*}(v_{2})$ and $\phi_{*}(v_{1})$ are defined.

In this case, $\phi(u_{1}) \neq \phi(u_{2})$, or exactly one of $\phi(u_{1})$ and $\phi(u_{2})$, say $\phi(u_{1})$, is not defined.

Assume $\phi(u_{1}) \neq \phi(u_{2})$. If $v_{2}$ has at least two PCF colors, then we assign a color not in $C_{1}$ to $v_{1}$, and a color not in $C_{2} \cup \{\phi(v_{1})\}$ to $v_{2}$. Suppose that $v_{2}$ has exactly one PCF color. If $c \geq 6$, then we assign a color not in $C_{1} \cup \{\phi_{*}(v_{2})\}$ to $v_{1}$, and a color not in $C_{2} \cup \{\phi(v_{1})\}$ to $v_{2}$ to ensure that $\phi_{*}(v_{1})$ and $\phi_{*}(v_{2})$ are defined in the resulting coloring. Note that $|C_{1} \cup \{\phi_{*}(v_{2})\}| \leq 5 \leq c-1$. If $c = 5$, then the inequality \eqref{eq-1} implies that $n_{2}(v_{2}) \geq 1$, so we first assign a color not in $C_{1}$ to $v_{1}$, and then we assign a color not in $C_{2} \cup \{\phi(v_{1})\}$ to $v_{2}$. This ensures that $\phi_{*}(v_{1})$ is defined in the resulting coloring. If $\phi_{*}(v_{2})$ is not defined at this point, then we assign a color not in $C_{w} = \phi(N_{3^{+}}(v_{2}))\cup \{\phi(v_{2}), \phi(w_{0}), \phi_{*}(w_{0})\}$ on a $2$-neighbor $w$ of $v_{2}$ so that $\phi_{*}(v_{1})$ and $\phi_{*}(v_{2})$ are both defined in the final coloring, where $w_{0}$ is the other neighbor of $w$. There is at least one feasible color for $w$ since 
\[
|C_{w}| \leq 3 + \left\lfloor \frac{n_{3^{+}}(v_{2})}{2} \right\rfloor \leq 3 + \left\lfloor\frac{c-n_{2}(v)}{4}\right\rfloor \leq c-1. \text{ (Note that $n_{2}(v) \geq 1$)}
\]

Assume $\phi(u_{1})$ is not defined. Then $u_{1}$ is a $2$-vertex in $G$. We assign a color not in $C_{1} \cup \{\phi_{*}(v_{2})\}$ to $v_{1}$, and we assign a color not in $C_{2} \cup \{\phi(v_{1})\}$ to $v_{2}$. Since $\phi(u_{1})$ is not defined, we have that $|C_{1} \cup \{\phi_{*}(v_{2})\}| \leq 3 + 1 \leq c-1$. We then assign a color not in $\{\phi(v_{1}), \phi(u_{2}), \phi(x_{1}), \phi_{*}(x_{1})\}$ to $u_{1}$ so that $\phi_{*}(v_{1})$ and $\phi_{*}(v_{2})$ are defined in the resulting coloring, where $x_{1}$ is the other neighbor of $u_{1}$.

In all these three cases, we obtain a PCF $c$-coloring of a graph obtained from $G$ by possibly deleting some $2$-vertices in $N_{2}(v_{1}) \cup N_{2}(v_{2})$. Finally, we color all the other vertices in $N_{2}(v_{1}) \cup N_{2}(v_{2})$ to obtain a PCF $c$-coloring of $G$.
\end{proof}

It is important to note that the discharging part is the same as that in~\cite{Cho2022a}. However, for the sake of completeness, we include it as an appendix with minor modifications and clarifications.

\section{Odd coloring}

In this section, we discuss some odd coloring problems. A $3^{+}$-vertex is \emph{easy} if it has either odd degree or a $2^{-}$-neighbor. For any vertex $v$, let $N_e(v)$ denote the set of easy neighbors of $v$ and let $n_e(v)=|N_e(v)|$. Let $Y$ be a subset of $V(G)$, and let $\phi$ be a proper $c$-coloring of $G - Y$. We say that $\phi$ is a \emph{semi-odd $c$-coloring} of $(G, Y)$ if every vertex in $G - Y - N(Y)$ has an odd color with respect to $\phi$.

An analogue of \autoref{Lem:semi-PCF} is the following lemma for semi-odd coloring.
\begin{lemma}\label{Lem:semi-ODD}
Assume $c\geq 5$ is an integer, and $v$ has odd degree or a $2^{-}$-neighbor. If $G$ has no odd $c$-coloring, but $(G, \{v\} \cup N_{1}(v) \cup N_2(v))$ has a semi-odd $c$-coloring, then $2d(v) \geq 2n_{1}(v) + n_2(v) + n_e(v) + c$.
\end{lemma}
\begin{proof}
Let $X = N(N_2(v)) \setminus (\{v\} \cup N_{2}(v))$. Suppose to the contrary that $2d(v) \leq 2n_{1}(v) + n_2(v) + n_e(v) + c-1$.
By the assumption, $(G, \{v\} \cup N_{1}(v) \cup N_2(v))$ has a semi-odd $c$-coloring $\phi$. We assign a color not in $C= \phi( X\cup N_{3^{+}}(v)) \cup \phi_{o}(N_{3^{+}}(v)\setminus N_e(v))$ to $v$.
Note that
\[
\begin{array}{rcl}
|C| \leq |X \cup N_{3^{+}}(v)| + |N_{3^{+}}(v)\setminus N_e(v)| & \leq &
(d(v) - n_{1}(v)) + (d(v) - n_{1}(v) - n_2(v) - n_e(v))\\ &=& 2d(v) - 2n_{1}(v) - n_2(v) - n_e(v)\\ & \leq & c-1,
\end{array}
\]
so there is at least one color to use on $v$. We subsequently proceed to color the vertices in $N_2(v)$ one by one. For each $2$-vertex $u$ in $N_2(v)$, we assign a color not in $\phi(N(u))\cup\phi_{o}(N(u))$ to $u$. Note that $|\phi(N(u)) \cup \phi_{o}(N(u))| \leq 4 \leq c-1$. Finally, for each $1$-vertex $u$ in $N_{1}(v)$, we assign a color not in $\phi(v) \cup \phi_{o}(v)$. At this point, we have a proper coloring of $G$, and since either $d(v)$ is odd or $N_{1}(v) \cup N_2(v)\neq \emptyset$, every vertex in $G$ has an odd color, except potentially those in $N_e(v)$. If a vertex $u \in N_e(v)$ does not have an odd color, then $d(u)$ must be even, implying that $u$ has a $2^{-}$-neighbor. For each such $u$, we can recolor a $2^{-}$-neighbor $x$ of $u$ with a color not in $\phi(N(x))\cup \phi_{o}(N(x))$, thereby guaranteeing an odd color for $u$. Note that $|\phi(N(x))\cup\phi_{o}(N(x))| \leq 4 \leq c-1$. Consequently, we obtain an odd $c$-coloring of $G$, leading to a contradiction.
\end{proof}

With the application of \autoref{Lem:semi-ODD}, we are able to provide a concise proof of \autoref{ODD-MAD-WY}, as detailed in \autoref{Appendix:B}.

\begin{theorem}\label{ODD-MAD-WY}
Let $c \geq 5$ be an integer, and let $G$ be a graph with $\mad(G) \leq \mad(\mathrm{SK}_{c+1}) = \frac{4c}{c+2}$. Then $\chi_{\mathrm{o}}(G) \geq c+1$ if and only if $G \in \mathcal{G}_{c}$.
\end{theorem}

\subsection{Odd 4-colorable sparse graphs}
\resetcounter
In this subsection, we devote to proving \autoref{Thm:Odd4colorable}. It is worth noting that this proof relies heavily on the one in~\cite[Theorem 1.3]{MR4533825}.

Let $G$ be a counterexample to \autoref{Thm:Odd4colorable} with the minimum number of vertices. Then $G$ is a graph with $\mad(G) \leq \frac{22}{9}$ and $G \notin \mathcal{H}$, but $G$ has no odd $4$-coloring. Observe that $G$ is connected and has at least $5$ vertices.

Recall that an \emph{$\ell$-thread} is a path on at least $\ell$ vertices, all having degree $2$ in $G$.

\begin{claim}\label{Claim:block}
No end-block of $G$ is a $5$-cycle.
\end{claim}
\begin{proof}
Assume that $G$ has an end-block that is a $5$-cycle $[u_{1}u_{2} \dots u_{5}]$, where $u_{2}, u_{3}, u_{4}$ and $u_{5}$ are $2$-vertices. Let $G' = G - \{u_{2}, u_{3}, u_{4}, u_{5}\}$. Note that $G'$ has at least two vertices, otherwise $G$ itself is a $5$-cycle, leading to a contradiction. Suppose that $G'$ is in $\mathcal{H}$. Then $G'$ consists of $t$ blocks, and each block is a $5$-cycle. It follows that $G$ consists of $t+1$ blocks, each being a $5$-cycle. This implies that $G$ belongs to $\mathcal{H}$, a contradiction. Hence, $G'$ is not in $\mathcal{H}$. By the minimality, $G'$ admits an odd $4$-coloring $\phi$. We can obtain an odd $4$-coloring of $G$ by greedily coloring $u_{2}, u_{5}, u_{3}$ and $u_{4}$ in this order. We first properly color $u_{2}$ and $u_{5}$ with the same color $\alpha$. Note that when coloring $u_3$ and $u_{4}$, there are at most three colors in  $\{\phi(u_{i-1}), \phi_{o}(u_{i-1}), \phi(u_{i+1}), \phi_{o}(u_{i+1})\}$ (consider the indices modulo $5$), so we can assign a color to each $u_i$, $i \in \{3, 4\}$, without conflict.
\end{proof}

\begin{claim}\label{Claim:mad_rc}
The followings do not appear in the counterexample $G$:

\begin{enumerate}[label = \rm\textbf{(\roman*)}]
    \item A $1$-vertex.
    \item For $\ell \in \{3, 4\}$, an end-block that is an $\ell$-cycle.
    \item A $4$-thread.
    \item A vertex of odd degree adjacent to a $2$-thread.
    \item A $3$-vertex with three $2$-neighbors.
\end{enumerate}
\end{claim}
\begin{proof} 
(i) 
Suppose to the contrary that $G$ contains a $1$-vertex $v$. Let $u$ be the unique neighbor of $v$. Let $S = \{v\}$. It follows from \autoref{Claim:block} that $G - S$ is not in $\mathcal{H}$. We obtain an odd 4-coloring $\phi$ of $G - S$. To extend $\phi$ to the entire graph $G$, we assign a color not in $\{\phi(u), \phi_{o}(u)\}$ to $v$.

(ii) 
Suppose to the contrary that $G$ has an $\ell$-cycle $[u_{1} \dots u_\ell]$ with $2$-vertices except $u_{1}$. Let $S = \{u_{2}, \dots, u_\ell\}$. It follows from \autoref{Claim:block} that $G - S$ is not in $\mathcal{H}$. Thus, we obtain an odd $4$-coloring $\phi$ of $G-S$. Note that the $\ell$-cycle is a block of $G$. For $\ell = 3$, we obtain an odd $4$-coloring of $G$ by greedily coloring $u_{2}, u_{3}$ in this order, where $\phi(u_{3})$ is a color not in $\{\phi(u_{2}), \phi(u_{1}), \phi_{o}(u_{1})\}$.

For $\ell = 4$, we obtain an odd $4$-coloring of $G$ by greedily coloring $u_{2}, u_{4}, u_{3}$ in this order. When coloring $u_{4}$, we assign a color not in $\{\phi(u_{2}), \phi(u_{1}), \phi_{o}(u_{1})\}$, while for $u_{3}$, we assign a color not in $\{\phi(u_{1}), \phi(u_{2}), \phi(u_{4})\}$ to it.

(iii) 
Suppose to the contrary that $G$ has a path $u_{1}v_{1}v_{2}v_{3}v_{4}u_{2}$, where $v_{1}v_{2}v_{3}v_{4}$ is a $4$-thread. Let $S=\{v_{1}, v_{2}, v_{3}, v_{4}\}$. It is obvious that $G - v_{1}$ is not in $\mathcal{H}$. Thus, we obtain an odd 4-coloring of $G - v_{1}$. Let $\phi$ denote the restriction of this coloring on $G-S$. Note that all $u_i$'s and $v_i$'s are distinct by (ii) and \autoref{Claim:block}. Then $v_{1}$ and $v_{4}$ are at distance $3$ in $G$.

If $\{\phi(u_{1}), \phi_{o}(u_{1})\}$ and $\{\phi(u_{2}), \phi_{o}(u_{2})\}$ are not disjoint, then there exists a color that can be used on both $v_{1}$ and $v_{4}$. Consequently, we can greedily color $v_2$ and $v_3$ to obtain an odd 4-coloring of $G$.

If $\{\phi(u_{1}), \phi_{o}(u_{1})\}$ and $\{\phi(u_{2}), \phi_{o}(u_{2})\}$ are disjoint, then we color $v_{1}$ with $\phi(u_{2})$ and $v_{4}$ with $\phi(u_{1})$. Now, we can greedily color $v_2$ and $v_3$ to obtain an odd 4-coloring of $G$.

(iv)
Suppose to the contrary that $G$ has a path $u_{1}v_{1}v_{2}u_{2}$, where $v_{1}v_{2}$ is a $2$-thread and $u_{1}$ is a vertex of odd degree. Let $S=\{v_{1}, v_{2}\}$. Similar to (iii), $G - v_{1}$ is not in $\mathcal{H}$, which allows us to obtain an odd 4-coloring of $G - v_{1}$. This yields a restriction $\phi$ on $G - S$. Note that $u_{1}\neq u_{2}$ by (ii). Since $u_{1}$ has odd degree, we can obtain an odd 4-coloring of $G$ by coloring $v_{2}$ with a color not in $\{\phi(u_{2}), \phi_{o}(u_{2}), \phi(u_{1})\}$ and $v_{1}$ with a color not in $\{\phi(v_{2}), \phi(u_{2}), \phi(u_{1})\}$.

(v)
Suppose to the contrary that $G$ has a 3-vertex $v$ that is adjacent to only $2$-neighbors $u_{1}, u_{2}, u_{3}$. Let $u'_{i}$ be the neighbor of $u_{i}$ that is not $v$ for each $i \in \{1, 2, 3\}$. Let $S=\{v, u_{1}, u_{2}, u_{3}\}$. Similarly, $G - v$ is not in $\mathcal{H}$, which allows us to obtain an odd 4-coloring of $G - v$. This yields a restriction $\phi$ on $G - S$. Now, we can obtain an odd 4-coloring of $G$ by coloring $v$ with a color not in $\{\phi(u'_{1}), \phi(u'_{2}), \phi(u'_3)\}$, and then coloring $u_{i}$ with a color not in $\{\phi(u'_{i}), \phi_{o}(u'_{i}), \phi(v)\}$ for each $i$ one by one.
\end{proof}

\begin{claim}\label{Claim:mad_threads}
Let $v$ be a $4^{+}$-vertex in $G$.
The following statements hold:
\begin{enumerate}[label = \rm\textbf{(\roman*)}]
    \item If $v$ has only $2$-neighbors, then $v$ is adjacent to at most $d(v)-2$ $2$-threads.
    \item If $v$ is adjacent to a $3$-thread, then $v$ is adjacent to at most $d(v)-4+\min\{n_{3^{+}}(v),1\}$ other $2$-threads.
\end{enumerate}
\end{claim}
\begin{proof}
Let $v$ be a $d$-vertex, and let $u_{1}, \dots, u_{d}$ be the neighbors of $v$. If $u_i$ is a $2$-vertex, let $x_i$ be the neighbor of $u_i$ other than $v$. For a $2$-thread $u_ix_i$, let $y_i$ be the neighbor of $x_i$ other than $u_i$. By \autoref{Claim:mad_rc}(ii), $N_{2}(v)$ is an independent set. For the same reason, $u_i$'s and $x_i$'s are all distinct, and $x_{i}$'s form an independent set.

(i) Suppose to the contrary that $u_ix_i$ is a $2$-thread for all $i\in\{2, \dots, d\}$.
Let $S = \{v, u_{1}, \dots, u_{d},$ $x_{2}, \dots, x_{d}\}$. Note that $G - v$ is not in $\mathcal{H}$, so we obtain an odd 4-coloring of $G - v$. Let $\phi$ denote the restriction of this coloring on $G-S$. Assume $[vu_{1}x_{1}x_{i}u_{i}]$ is a $5$-cycle with three $2$-vertices $u_{1}, u_{i}, x_{i}$. By \autoref{Claim:mad_rc}(iv), $x_{1}$ is a $4^{+}$-vertex with even degree in $G$. We color all the vertices in $N(x_{1}) \cap \{x_{2}, \dots, x_{d}\}$ with a color $\beta \neq \phi(x_{1})$ such that the resulting coloring $\pi$ satisfies $\beta \neq \pi_{o}(x_{1})$ if $\pi_{o}(x_{1})$ is defined. This can be done since $\phi$ is an odd coloring of $G- S$. If there is no $5$-cycle $[vu_{1}x_{1}x_{i}u_{i}]$, let $\pi = \phi$. Note that all the neighbors of $x_{1}$ except $u_{1}$ are colored under $\pi$. We color $v$ with $\pi_{o}(x_{1})$. For each uncolored vertex $x_{i}$, we give a color not in $\{\pi(y_i), \pi_{o}(y_i), \pi(v)\}$ to $x_i$ one by one, then assign a color not in $\{\pi(x_i), \pi(y_i), \pi(v)\}$ to $u_i$ for each $i\in\{2, \dots, d\}$. At this point, all neighbors of $v$ except $u_{1}$ are colored, so $\pi_{o}(v)$ is defined. Now, to complete an odd 4-coloring of $G$, we assign a color not in $\{\pi(v), \pi_{o}(v), \pi(x_{1})\}$ to $u_{1}$. This is a contradiction.

(ii) Let $u_{1}x_{1}y_{1}$ be a $3$-thread, and let $z_{1}$ be the neighbor of $y_{1}$ that is not $x_{1}$. It follows from \autoref{Claim:mad_rc}(ii) that $v \neq z_{1}$. Let $S= \{v,y_{1}\}\cup N_2(v) \cup\{x_i\mid u_ix_i~\text{ is a $2$-thread}\}$. Similarly, $G - v$ is not in $\mathcal{H}$, so we obtain an odd $4$-coloring of $G - v$. Let $\phi$ denote the restriction of this coloring on $G-S$.

Suppose to the contrary that $v$ is adjacent to $d-3+\min\{n_{3^{+}}(v), 1\}$ $2$-threads besides $u_{1}x_{1}$.
We have two possible cases to consider: (a) $u_ix_i$ is a $2$-thread for every $i\in\{2,\dots, d-1\}$, and $u_d$ is the only $3^{+}$-neighbor of $v$; (b) $u_ix_i$ is a $2$-thread for every $i\in\{2,\dots, d-2\}$, and $v$ has only $2$-neighbors. In the former case, let $X=\{\phi(u_d), \phi_{o}(u_d)\}$, and in the latter case, let $X=\{\phi(x_{d-1}), \phi(x_d)\}$. By \autoref{Claim:mad_rc}(ii), $y_{1} \neq u_{i}$, where $u_{i} \in N_{2}(v)$. By \autoref{Claim:mad_rc}(ii) and (iii), if $u_{i}x_{i}$ is a $2$-thread, then $y_{1} \neq x_{i}$ and $y_{1}u_{i} \notin E(G)$. We assign a color $\gamma \neq \phi(z_{1})$ to all vertices in $N_{2}(z_{1}) \cap (S\setminus \{v, y_{1}\})$ such that the resulting coloring $\psi$ satisfies $\psi_{0}(z_{1}) = \gamma$ if $\psi_{0}(z_{1})$ exists. If $N_{2}(z_{1}) \cap (S\setminus \{v, y_{1}\}) = \emptyset$, we let $\psi = \phi$. For convenience, in both cases, if $\psi_{0}(z_{1})$ exists, we let $\gamma = \psi_{0}(z_{1})$; otherwise, we select $\gamma$ to be a fixed color different from $\psi(z_{1})$.

Now, we assign a color not in $X\cup\{\gamma\}$ to $v$, and then color the vertices in $S\setminus\{v,u_{1},x_{1},y_{1}\}$ as follows. For each $2$-thread $u_ix_i$ where $i\geq 2$, use a color not in $\{\psi(y_i), \psi_{o}(y_i),\psi(v)\}$ on $x_i$ if $x_{i}$ is uncolored, and for each $2$-vertex $u_i$ where $i \geq 2$, we assign a color not in $\{\psi(v), \psi(x_i), \psi_{o}(x_i)\}$ to $u_i$ if $u_{i}$ is uncolored. With these steps, all vertices except $u_{1}, x_{1}, y_{1}$ are colored, and all vertices except $v, u_{1}, x_{1}, y_{1}, z_{1}$ satisfy the oddness property.

Finally, we color the remaining uncolored vertices among $u_{1},x_{1},y_{1}$ as follows. Assign a color not in $\{\psi(v), \psi_{o}(v), \gamma\}$ on $u_{1}$, and assign a color not in $\{\psi(u_{1}), \psi(z_{1}), \gamma\}$ on $y_{1}$. Note that $\gamma$ is not used on $v, u_{1}, y_{1}, z_{1}$. Thus, we assign $\gamma$ on $x_{1}$ to extend $\phi$ to the entire graph $G$, which is a contradiction.
\end{proof}

If a $2$-vertex $u$ is on a thread that is adjacent to a $3^{+}$-vertex $v$, then we say that $u$ is \emph{close} to $v$, and $v$ is referred to as a \emph{sponsor} of $u$.

\begin{claim}\label{Claim:max_numclose}
Let $v$ be a $4^{+}$-vertex of $G$.
If $v$ has odd degree, then it has at most $d(v)$ close $2$-vertices.
If $v$ has even degree, then it has at most $3d(v)-5$ close $2$-vertices.
\end{claim}
\begin{proof}
If $v$ has odd degree, then it cannot be adjacent to a $2$-thread by \autoref{Claim:mad_rc}(iv), so it has at most $d(v)$ close $2$-vertices.

Now, assume $v$ has even degree. If $v$ has a $3^{+}$-neighbor, then $v$ has at most $\max\{3(d(v)-2)+1, 2(d(v)-1)\} = 3d(v)-5$ close $2$-vertices by \autoref{Claim:mad_rc}(iii) and \autoref{Claim:mad_threads}(ii).
If $v$ only has $2$-neighbors, then $v$ is adjacent to at most $\max\{3(d(v)-3)+3, 2(d(v)-2)+2\}=3d(v)-6$ close $2$-vertices by \autoref{Claim:mad_rc}(iii) and \autoref{Claim:mad_threads}.
Hence, $v$ has at most $3d(v)-5$ close $2$-vertices.
\end{proof}

Let $\mu(v)$ be the initial charge of a vertex $v$, with $\mu(v)=d(v)$.
Let $\mu^{*}(v)$ denote the final charge of vertex $v$ after applying the following discharging rule:

\begin{enumerate}[label = \textbf{R\arabic*.}, ref = R\arabic*]
    \item
    Each $3^{+}$-vertex sends charge $\frac{2}{9}$ to its close $2$-vertices.
\end{enumerate}

Note that the total charge is preserved during the discharging process. Consider an arbitrary vertex $v$. We will show that $\mu^{*}(v) \geq \frac{22}{9}$. Note that $G$ contains no $1$-vertices by \autoref{Claim:mad_rc}(i). If $v$ is a $2$-vertex, then it receives charge $\frac{2}{9}$ from each of its two sponsors according to the rule. Hence, $\mu^{*}(v)=2+ \frac{2}{9} \times 2=\frac{22}{9}$.

For a $3$-vertex $v$, at most two of its neighbors are $2$-vertices by \autoref{Claim:mad_rc}(v). Furthermore, $v$ cannot be adjacent to a $2$-thread by \autoref{Claim:mad_rc}(iv). Thus, $v$ has at most two close $2$-vertices, so it sends charge $\frac{2}{9}$ at most twice by the rule. Hence, $\mu^{*}(v) \geq 3 - \frac{2}{9} \times 2 = \frac{23}{9} > \frac{22}{9}$.

If $v$ is a $4^{+}$-vertex of even degree, then it has at most $3d(v) - 5$ close $2$-vertices by \autoref{Claim:max_numclose}. Therefore, it sends charge $\frac{2}{9}$ at most $3d(v)-5$ times by the rule. Hence, $\mu^{*}(v) \geq d(v)-\frac{2}{9}\times(3d(v)-5)\geq\frac{22}{9}$ since $d(v)\geq 4$.

If $v$ is a $5^{+}$-vertex of odd degree, then it has at most $d(v)$ close $2$-vertices by \autoref{Claim:max_numclose}.
Therefore, it sends charge $\frac{2}{9}$ at most $d(v)$ times by the rule.
Hence, $\mu^{*}(v) \geq d(v)-\frac{2}{9}\times d(v)> \frac{22}{9}$ since $d(v)\geq 5$.

The sum of the initial charge is at most $\frac{22}{9}|V(G)|$, and the sum of the final charge is at least $\frac{22}{9}|V(G)|$. Hence, every vertex $v$ has $\mu^{*}(v) = \frac{22}{9}$. By the above discussion, $G$ only consists of $2$-vertices and $4$-vertices. Moreover, every $4$-vertex has precisely $3d(v) - 5 = 7$ close $2$-vertices. By \autoref{Claim:max_numclose}, every $4$-vertex sponsors two $3$-threads and one other $1$-thread. 

We define $G_{1}$ as the subgraph induced by the edges on $3$-threads in $G$, and $G_{2}$ as the subgraph induced by the edges in $E(G)\setminus E(G_{1})$. It follows that $G_{1}$ and $G_{2}$ are spanning $2$-factors of $G$. We split the vertex set $V(G)$ into two disjoint sets: $U$, where $U$ is the set of all $4$-vertices of $G$, and $W$, which is the set of all $2$-vertices of $G$. Let $G_{2}^{*}$ be the graph obtained from $G_{2}$ by adding edges $uw$ if $uvw$ is a $1$-thread in $G$ and $uw \notin E(G)$. Since every $4$-vertex in $G$ sponsors exactly one $1$-thread, $G_{2}^{*}$ is a subcubic graph and is $4$-colorable (by Brooks' theorem). 

Clearly, every proper $4$-coloring of $G_{2}^{*}$ yields a proper $4$-coloring of $G_{2}$ such that every vertex in $W \cap V(G_{2})$ satisfies the oddness condition. Let $\phi$ be a such coloring. We extend this coloring to an odd $4$-coloring of $G$, which leads to a contradiction. Since $G_{1}$ is a spanning $2$-factor of $G$, we give an orientation of $G_{1}$ in which every component of $G_{1}$ is a directed cycle. For each directed $3$-thread \overrightharp{$ux_{1}y_{1}z_{1}v$}, we refer to $z_{1}$ as the \emph{friend} of $v$. Note that every $4$-vertex in $G$ has exactly one friend. We firstly color all the friends of the $4$-vertices. For each directed $3$-thread \overrightharp{$ux_{1}y_{1}z_{1}v$}, if $\phi(u) \neq \phi(v)$, then we assign $\phi(u)$ to $z_{1}$; otherwise, if $\phi(u) = \phi(v)$, we assign a color different from $\phi(v)$ to  $z_{1}$. At this point, every $4$-vertex in $G$ has exactly one uncolored neighbor. Next, we color the remaining $2$-vertices on $3$-threads. For each directed $3$-thread \overrightharp{$ux_{1}y_{1}z_{1}v$}, if $\phi(u) = \phi(z_{1})$, we assign a color not in $\{\phi(u), \phi_{o}(u)\}$ to $x_{1}$, and then assign a color not in $\{\phi(u), \phi(x_{1}), \phi(v)\}$ to $y_{1}$; otherwise, if $\phi(u) \neq \phi(z_{1})$, then $\phi(u) = \phi(v)$, so we assign a color not in $\{\phi(u), \phi_{o}(u), \phi(z_{1})\}$ to $x_{1}$, and assign a color not in $\{\phi(u), \phi(x_{1}), \phi(z_{1})\}$ to $y_{1}$. This yields an odd $4$-coloring of $G$, a contradiction.

\subsection{Odd 6-colorable sparse planar graphs}
\resetcounter

In this subsection, we prove \autoref{Thm:Odd6colorable}.

Let $G$ be a counterexample to \autoref{Thm:Odd6colorable} with $|V(G^{*})| + |E(G^{*})|$ minimized, where $G^{*}$ is induced by all the $3^{+}$-vertices in $G$; subject to this, $|V(G)|$ is minimized. Observe that every proper induced subgraph of $G$ has an odd $6$-coloring. We fix a plane embedding of $G$. So, $G$ is a plane graph without $4^{-}$-faces adjacent to $7^{-}$-faces that has no odd $6$-coloring. Then the following claim holds immediately by \autoref{Lem:semi-ODD}.

\begin{claim}\label{Claim:Planar_rc}
The following configurations do not appear in the counterexample $G$:
\begin{enumerate}[label = \rm\textbf{(\roman*)}]
    \item A $1$-vertex.
    \item A $2$-thread.
    \item A $3$-vertex with either a $2$-neighbor or an easy neighbor.
    \item A $4$-vertex with three $2$-neighbors.
\end{enumerate}
\end{claim}

\begin{claim}\label{3face}
A $2$-vertex is not contained in a $3$-cycle.
\end{claim}
\begin{proof}
Let $[xyz]$ be a $3$-cycle with $d(x) = 2$. By the minimality, $G - x$ has an odd $6$-coloring $\phi$. Since $y$ and $z$ are adjacent, we have that $\phi(y) \neq \phi(z)$. We extend $\phi$ by assigning a color not in $\{\phi(y), \phi_{o}(y), \phi(z), \phi_{o}(z)\}$ to the $2$-vertex $x$. It is clear that the resulting coloring is an odd $6$-coloring of $G$, a contradiction.
\end{proof}

\begin{claim}\label{adj-easy}
There are no adjacent easy vertices.
\end{claim}
\begin{proof}
Suppose, to the contrary, that $x$ and $y$ are two adjacent easy vertices. Let $G'$ be the graph obtained by inserting a vertex $v$ on $xy$ of $G$. By the minimality, $G'$ has an odd $6$-coloring. Let $\phi$ be the restriction of this coloring on $V(G)$. Note that $\phi$ is a proper coloring of $G$. Observe that every vertex except $x$ and $y$ has an odd color. If $x$ or $y$ has odd degree, then it has an odd color. If $x$ (and/or $y$) has a $2$-neighbor $x'$ (and/or $y'$) but it has no odd color, then we assign a new color to $x'$ (and/or $y'$) with a color that is not in $\{\phi(x), \phi(x''), \phi_{o}(x'')\}$ (and/or $\{\phi(y), \phi(y''), \phi_{o}(y'')\}$) to ensure that $x$ (and/or $y$) has an odd color, where $N(x') = \{x, x''\}$ and $N(y') = \{y, y''\}$.
\end{proof}

\begin{claim}\label{45face}
If a $4$- or $5$-face is incident with a $2$-vertex $u_{1}$, then it is incident with only one $3^{-}$-vertex, namely $u_{1}$.
\end{claim}
\begin{proof}
Suppose that $[u_{1}u_{2}u_{3}u_{4}]$ is a $4$-face with $d(u_{1}) = 2$. By \autoref{Claim:Planar_rc}(ii) and (iii), $u_{2}$ and $u_{4}$ are easy $4^{+}$-vertices. By \autoref{adj-easy}, $u_{3}$ has even degree. Assume $u_{3}$ is a $2$-vertex. By the minimality, $G - u_{3}$ has an odd $6$-coloring $\phi$. Since $u_{2}$ and $u_{4}$ are the two neighbors of the $2$-vertex $u_{1}$, we have that $\phi(u_{2}) \neq \phi(u_{4})$. We assign a color not in $\{\phi(u_{2}), \phi_{o}(u_{2}), \phi(u_{4}), \phi_{o}(u_{4})\}$ to $u_{3}$, obtaining an odd $6$-coloring of $G$. Hence, $u_{3}$ is a non-easy $4^{+}$-vertex.

Suppose that $[u_{1}u_{2}u_{3}u_{4}u_{5}]$ is a $5$-face with $d(u_{1}) = 2$. Similarly, we can prove $u_{2}$ and $u_{5}$ are easy $4^{+}$-vertices. By \autoref{adj-easy}, neither $u_{3}$ nor $u_{4}$ is an easy vertex. Hence, $u_{3}$ and $u_{4}$ are non-easy $4^{+}$-vertices.
\end{proof}

A $5$-face is \emph{bad} if it is incident with a $2$-vertex, otherwise, we refer to it as \emph{good}.

For a face $f$, let $d(f)$ denote the degree of $f$, which is the length of the boundary of $f$. We define the initial charge $\mu(z)$ for any vertex or face $z$ as follows: for each vertex $v$, let $\mu(v)=d(v)-6$, and for each face $f$, let $\mu(f)=2d(f)-6$.
Let $\mu^{*}(z)$ be the final charge of $z$ after applying the following discharging rules:

\begin{enumerate}[label = \textbf{R\arabic*.}, ref = R\arabic*]
    \item\label{6rule:2vx} Each face sends charge $2$ to each incident $2$-vertex.
    \item\label{6rule:3vx} Each $4$- or $5$-face sends charge $1$ to each incident $3$-vertex.
    \item\label{6rule:bad5} Each bad face sends charge $\frac{1}{2}$ to each incident $4^{+}$-vertex.
    \item\label{6rule:good5} Each good face sends charge $1$ to each incident easy $4^{+}$-vertex, and $\frac{1}{2}$ to each incident non-easy $4^{+}$-vertex.
    \item\label{6rule:6face} For a $3^{+}$-vertex $v$ on a $6^{+}$-face $f$, if $x, v, z$ are three consecutive vertices on the boundary of $f$, then $f$ sends charge $\left(\frac{2d(f) - 6}{d(f)} - 1\right) \times |\{x,z\} \cap N_2(v)| + \frac{d(f) - 3}{d(f)} \times |\{x,z\} \setminus N_{2}(v)|$ to $v$.
    \item\label{6rule:noneasy4vertex1}For each non-easy $4^{+}$-vertex $v$ incident with a 3-face $f$, if $vw\in b(f)$ and $d(w)=3$, then $v$ sends charge $\frac{1}{4}$ to $w$, where $b(f)$ denotes the boundary of $f$.
    \item\label{6rule:noneasy4vertex2} For each non-easy $4^{+}$-vertex $v$ incident with a $4^{-}$-face $f$, if $vw\in b(f)$ and $w$ is an easy 4-vertex, then $v$ sends charge $\frac{1}{8}$ to $w$.
\end{enumerate}

Note that the total charge is preserved during the discharging process.

\begin{claim}
Each face has non-negative final charge.
\end{claim}
\begin{proof}
Let us analyze the final charge for each type of face. Suppose $f$ is a $6^{+}$-face.
We first send charge $\frac{2d(f) - 6}{d(f)}$ to each incident edge, which is possible since $\mu^{*}(f)=2d(f)-6 - d(f) \times \frac{2d(f) - 6}{d(f)} = 0$. Since $d(f) \geq 6$, each incident edge receives at least $1$ unit of charge.
If an edge $e$ is incident with a $2$-vertex $x$, then $e$ sends charge $1$ to $x$ by \ref{6rule:2vx}, and sends charge $\frac{2d(f) - 6}{d(f)} - 1$ to the other incident vertex (note that the other vertex incident with $e$ is a $4^{+}$-vertex),
and if $e$ is not incident with a $2$-vertex, then $e$ sends charge $\frac{d(f) - 3}{d(f)}$ to each vertex incident with $e$.
Consider a vertex $v$.
If $v$ is a $2$-vertex, then it will receive charge $2$ from $f$ since both edges incident with $v$ will send charge $1$ to $v$.
If $v$ is a $3$-vertex, then it will receive charge $\frac{2d(f) - 6}{d(f)}$ from $f$ since $v$ is not adjacent to a $2$-vertex by \autoref{Claim:Planar_rc}(iii).
If $v$ is a $4^{+}$-vertex such that $x, v, z$ are three consecutive vertices on the boundary of $f$, then $v$ will receive charge $\left(\frac{2d(f) - 6}{d(f)} - 1\right) \times |\{x,z\} \cap N_2(v)|+ \frac{d(f) - 3}{d(f)} \times |\{x,z\} \setminus N_2(v)|$ from $f$.

Suppose $f$ is a bad face. By \autoref{45face}, $f$ is incident with exactly one $2$-vertex, and it is not incident with a $3$-vertex. Then $f$ sends charge $2$ to its incident $2$-vertex and charge $\frac{1}{2}$ to each of its incident $4^{+}$-vertices by \ref{6rule:2vx} and \ref{6rule:bad5}.
Hence, $\mu^{*}(f) = 4-2-4\times\frac{1}{2}=0$.

Suppose $f$ is a good face. Then it is not incident with a $2$-vertex.
By \autoref{adj-easy}, $f$ is incident with at most two easy vertices.
Thus, $f$ sends charge $1$ to each of its incident easy vertices and $\frac{1}{2}$ to each of the other incident $4^{+}$-vertices by \ref{6rule:3vx} and \ref{6rule:good5}.
Hence, $\mu^{*}(f) \geq 4 - 2 \times 1 - 3 \times \frac{1}{2} > 0$.

Suppose $f$ is a $4$-face. If $f$ is incident with a $2$-vertex, then it is incident with exactly one $3^{-}$-vertex by \autoref{45face}, and it sends charge 2 to the unique $2$-vertex by \ref{6rule:2vx}, so $\mu^{*}(f) = 2 - 2 = 0$. If $f$ is not incident with a $2$-vertex, then it is incident with at most two $3$-vertices by \autoref{adj-easy}, so $\mu^{*}(f) \geq 2 - 2 \times 1 = 0$ by \ref{6rule:3vx}.

Suppose $f$ is a $3$-face. Note that a $2$-vertex is not incident with a $3$-face by \autoref{3face}. Then $f$ does not send any charge to vertices, so $\mu^{*}(f) = 0$.
\end{proof}

\begin{claim}
Each vertex has non-negative final charge.
\end{claim}
\begin{proof}
Consider a vertex $v$. As shown in \autoref{Claim:Planar_rc}(i), $v$ is not a $1$-vertex.
If $v$ is a $2$-vertex, then it always receives charge $2$ twice from its incident faces by \ref{6rule:2vx}.
Hence, $\mu^{*}(v)=-4+2\times2=0$.
Let $v$ be a $3$-vertex. By \autoref{Claim:Planar_rc}(iii), $v$ is adjacent to three non-easy $4^{+}$-vertices. If $v$ is not incident with a $3$-face, then it always receives charge at least $1$ three times from its incident faces by \ref{6rule:3vx} and \ref{6rule:6face}.
Hence, $\mu^{*}(v)=-3+3\times1=0$. Now, let $v$ be incident with a $3$-face $uvw$. Then it is incident with two $8^{+}$-faces, and receives charge at least $\frac{2 \times 8 - 6}{8} = \frac{5}{4}$ from each incident $8^{+}$-face by \ref{6rule:6face}. Since $u$ and $w$ are both non-easy $4^{+}$-vertices, while $v$ is a $3$-vertex, we have that $v$ receives charge $\frac{1}{4}$ from each of $u$ and $w$ by \ref{6rule:noneasy4vertex1}. Hence, $\mu^{*}(v) \geq -3 + 2 \times \frac{5}{4}+ 2 \times \frac{1}{4} = 0$.

For a $4^{+}$-vertex $v$ on a $5^{+}$-face $f$ where $x, v, z$ are three consecutive vertices on the boundary of $f$, if $f$ sends charge less than $\frac{1}{2}$ to $v$ by the rules, then $f$ must be a $6^{+}$-face with both $x$ and $z$ being $2$-vertices.

Let $v$ be a $4$-vertex with neighbors $u_{1},u_{2},u_{3}, u_{4}$ in a cyclic order around $v$ and let $f_i$ be the face incident with $u_ivu_{i+1}$.

First suppose that $v$ is an easy 4-vertex. By \autoref{adj-easy}, each neighbor of $v$ is a non-easy vertex. By the definition of easy vertex, we have that $v$ has at least one $2$-neighbor, say $u_{1}$. \autoref{Lem:semi-ODD} implies $n_{2}(v) + n_{e}(v) \leq 2$. We consider two cases (a) and (b). 

(a) Suppose that $v$ is incident with four $5^{+}$-faces. If $v$ receives charge at least $\frac{1}{2}$ from each incident face, then $\mu^{*}(v) \geq -2 + 4 \times \frac{1}{2} = 0$. So we may assume that $v$ receives charge less than $\frac{1}{2}$ from $f_{1}$. Then $f_{1}$ is a $6^{+}$-face, where both $u_{1}$ and $u_{2}$ are $2$-vertices, and both $u_{3}$ and $u_{4}$ are non-easy $4^{+}$-vertices. Hence, $f_3$ is either a good face or a $6^{+}$-face, which sends charge at least $1$ to $v$ by~\ref{6rule:good5} or \ref{6rule:6face}. Each of $f_2$ and $f_{4}$ sends charge at least $\frac{1}{2}$ to $v$ by \ref{6rule:bad5}, \ref{6rule:good5}, or \ref{6rule:6face}. Hence, $\mu^{*}(v) \geq -2+1+2\times\frac{1}{2}=0$. 

(b) Suppose that $v$ is incident with a $4^{-}$-face. Then $v$ is incident with at least two $8^{+}$-faces. Next, we consider three subcases (i), (ii) and (iii). 

---(i) If $n_2(v)=1$, then $v$ receives charge at least $(\frac{2\times 8 - 6}{8} - 1+\frac{8 - 3}{8}) + 2 \times \frac{8 - 3}{8} = \frac{17}{8}$ from the two $8^{+}$-faces by \ref{6rule:6face}. Hence, $\mu^{*}(v) \geq -2 + \frac{17}{8} > 0$. 

---(ii) Suppose $n_2(v)=2$ and $d(u_{1})=d(u_3)=2$. Without loss of generality, assume $f_{1}$ is a $4^{-}$-face (in fact, $f_{1}$ is a $4$-face by \autoref{3face}). Then $v$ receives charge at least $(\frac{2\times 8 - 6}{8} - 1+\frac{8 - 3}{8})\times 2 = \frac{7}{4}$ from $f_2$ and $f_{4}$ by \ref{6rule:6face}, and it receives charge $\frac{1}{8}$ from $u_2$ by \ref{6rule:noneasy4vertex2}. By \autoref{3face}, $f_3$ is a $4^{+}$-face. Then $v$ receives charge at least $\frac{1}{2}$ from $f_{3}$ or charge $\frac{1}{8}$ from $u_{4}$. Hence, $\mu^{*}(v) \geq -2 + \frac{7}{4}+\frac{1}{8} + \min\{\frac{1}{8}, \frac{1}{2} \}= 0$. 

---(iii) Suppose $n_2(v)=2$ and $d(u_{1})=d(u_{2})=2$. By \autoref{3face} and \ref{45face}, $f_{1}$ is a $6^{+}$-face. If $f_3$ is a $4^{-}$-face, then each of $f_{2}$ and $f_{4}$ is an $8^{+}$-face. Hence, $v$ receives charge $\frac{1}{8}$ from each of $u_{3}$ and $u_{4}$ by \ref{6rule:noneasy4vertex2}, and it receives charge $(\frac{2\times 8 - 6}{8} - 1+\frac{8 - 3}{8})\times 2 = \frac{7}{4}$ from $f_{2}$ and $f_{4}$ by \ref{6rule:6face}, thus $\mu^{*}(v) \geq -2 + \frac{7}{4} + \frac{1}{8}\times 2= 0$. If $f_{1}$ is an $8^{+}$-face and $f_2$ is a $4^{-}$-face, then $v$ receives charge at least $(\frac{2\times 8 - 6}{8} - 1)\times 2 = \frac{1}{2}$ from $f_{1}$, charge at least $\frac{8 - 3}{8}\times 2 = \frac{5}{4}$ from $f_{3}$, charge $\frac{1}{8}$ from $u_{3}$, and charge $\frac{1}{8}$ from $u_{4}$ (when $f_{4}$ is a $4$-face) or at least $\frac{1}{2}$ from $f_{4}$ (when $f_{4}$ is a $5^{+}$-face), it follows that $\mu^{*}(v) \geq -2 + \frac{1}{2} + \frac{5}{4} + \frac{1}{8} + \min\{\frac{1}{8},\frac{1}{2}\}=0$.

Now suppose that $v$ is a non-easy 4-vertex. By the definition of non-easy vertex, we have that $v$ has no 2-neighbor. Then each $u_i$ is a $3^{+}$-vertex. (a) Suppose that $v$ is incident with four $5^{+}$-faces. Then each $f_i$ sends charge at least $\frac{1}{2}$ to $v$ by \ref{6rule:good5} and \ref{6rule:6face}. Hence, $\mu^{*}(v) \geq -2+4\times\frac{1}{2}=0$. (b) Suppose that $v$ is incident with a $4^{-}$-face, say $f_{1}$. Then $d(f_{2}) \geq 8$ and $d(f_{4}) \geq 8$. Then each of $f_2$ and $f_{4}$ sends charge $\frac{2\times 8-6}{8}=\frac{5}{4}$ to $v$ by \ref{6rule:6face}. By the discharging rules, if $v$ sends charge $\frac{1}{4}$ to a vertex in $\{u_{1}, u_{2}\}$, then $f_{1}$ is a $3$-face, and the other vertex on $f_{1}$ is neither a $3$-vertex nor an easy $4$-vertex by \autoref{adj-easy}. Then $v$ sends charge at most $\max\{\frac{1}{8} \times 2, \frac{1}{4}\} = \frac{1}{4}$ in total to $u_{1}$ and $u_{2}$. Similarly, $v$ sends charge at most $\frac{1}{4}$ in total to $u_{3}$ and $u_{4}$. Hence, $\mu^{*}(v) \geq -2 + \frac{5}{4} \times 2 - \frac{1}{4} \times 2 = 0$.

If $v$ is a $5$-vertex, then $n_2(v)+n_e(v) \leq 4$ by \autoref{Lem:semi-ODD}. If $v$ is not incident with a $4^{-}$-face, then $v$ receives charge less than $\frac{1}{2}$ at most three times from its incident faces.
Hence, $v$ receives charge at least $\frac{1}{2}$ twice from its incident faces, so by the rules, $\mu^{*}(v) \geq -1 + 2 \times \frac{1}{2} = 0$. Suppose that $v$ is incident with a $4^{-}$-face. Similar to the above, $v$ is incident with at least two $8^{+}$-faces, and $v$ receives at least $2 \times (\frac{2 \times 8 - 6}{8} -1) = \frac{1}{2}$ from each incident $8^{+}$-face. Hence, $\mu^{*}(v) \geq -1 + 2 \times \frac{1}{2} = 0$.

Let $v$ be a $6^{+}$-vertex with neighbors $u_{1}, u_{2}, \dots, u_{d}$ in a cyclic order, and let $f_{i}$ be the face incident with $u_{i}vu_{i+1}$. Assume $f_{i}$ is a $4^{-}$-face. Then both $f_{i-1}$ and $f_{i+1}$ are $8^{+}$-faces. If $v$ sends a positive charge to $u_{i}$, then $f_{i-1}$ sends charge at least $\frac{2 \times 8 - 6}{8} - 1 = \frac{1}{4}$ to $v$ via the edge $vv_{i}$. Similarly, if $v$ sends a positive charge to $u_{i+1}$, then $f_{i+1}$ sends charge at least $\frac{2 \times 8 - 6}{8} - 1 = \frac{1}{4}$ to $v$ via the edge $vv_{i+1}$. Note that $v$ sends charge at most $\frac{1}{4}$ to each neighbor by \ref{6rule:noneasy4vertex1} and \ref{6rule:noneasy4vertex2}. Hence, $\mu^{*}(v)\geq \mu(v)= d(v)-6\geq 0$.
\end{proof}

Therefore, every vertex and every face of $G$ has non-negative final charge, while the initial charge sum is at most $-12$ by Euler's formula. This is a contradiction.


\section*{Acknowledgements}
Xiaojing Yang was supported by National Natural Science Foundation of China (12101187).


\begin{thebibliography}{10}

\bibitem{MR4499342}
Y.~Caro, M.~Petru\v{s}evski and R.~\v{S}krekovski, Remarks on proper
  conflict-free colorings of graphs, Discrete Math. 346~(2) (2023) 113221.

\bibitem{Cho2022a}
E.-K. Cho, I.~Choi, H.~Kwon and B.~Park, Proper conflict-free coloring of
  sparse graphs, arXiv:2203.16390, 2022.

\bibitem{MR4533825}
E.-K. Cho, I.~Choi, H.~Kwon and B.~Park, Odd coloring of sparse graphs and
  planar graphs, Discrete Math. 346~(5) (2023) 113305.

\bibitem{Cranston2022}
D.~W. Cranston, Odd colorings of sparse graphs, arXiv:2201.01455, 2022.

\bibitem{MR4537616}
D.~W. Cranston, M.~Lafferty and Z.-X. Song, A note on odd colorings of 1-planar
  graphs, Discrete Appl. Math. 330 (2023) 112--117.

\bibitem{Dujmovic2022}
V.~Dujmovi\'{c}, P.~Morin and S.~Odak, Odd colourings of graph products,
  arXiv:2202.12882,   2022.

\bibitem{MR4493821}
I.~Fabrici, B.~Lu\v{z}ar, S.~Rindo\v{s}ov\'{a} and R.~Sot\'{a}k, Proper
  conflict-free and unique-maximum colorings of planar graphs with respect to
  neighborhoods, Discrete Appl. Math. 324 (2023) 80--92.

\bibitem{MR4654344}
R.~Hickingbotham, Odd colourings, conflict-free colourings and strong colouring
  numbers, Australas. J. Combin. 87~(1) (2023) 160--164.

\bibitem{MR4564916}
R.~Liu, W.~Wang and G.~Yu, 1-planar graphs are odd 13-colorable, Discrete Math.
  346~(8) (2023) 113423.

\bibitem{Metrebian2022}
H.~Metrebian, Odd colouring on the torus, arXiv:2205.04398,   2022.

\bibitem{Niu2022}
B.~Niu and X.~Zhang, An improvement of the bound on the odd chromatic number of
  1-planar graphs, in: Algorithmic Aspects in Information and Management,
  Springer International Publishing, 2022, pp. 388--393.

\bibitem{MR4559435}
J.~Petr and J.~Portier, The odd chromatic number of a planar graph is at most
  8, Graphs Combin. 39~(2) (2023) 28.

\bibitem{MR4467654}
M.~Petru\v{s}evski and R.~\v{S}krekovski, Colorings with neighborhood parity
  condition, Discrete Appl. Math. 321 (2022) 385--391.

\end{thebibliography}

\appendix
\section*{Appendix}
\resetcounter

\section{Discharging part for \autoref{PCF-MAD-WY}}\label{Appendix:A}
Let $\mu(v)$ be the initial charge of a vertex $v$, with $\mu(v) = d(v)$. We define $\mu^{*}(v)$ as the final charge of vertex $v$ after applying the discharging rules:
\subsection{Five colors}\label{sub:c5}

The discharging rules are as follows:
\begin{enumerate}[label = \textbf{R\arabic*.}, ref = R\arabic*]
    \item\label{rule:c=5_1}
    Each $3^{+}$-vertex sends charge $\frac{3}{7}$ to each of its $2$-neighbors.
    \item\label{rule:c=5_2}
    Each $4^{+}$-vertex sends charge $\frac{4}{21}$ to each of its $3$-neighbors $u$ with $n_2(u) \geq 1$, and to each of its $4$-neighbors $w$ with $n_2(w) \geq 3$.
\end{enumerate}

Now, let us analyze the final charge $\mu^{*}(v)$ for any vertex $v$. We aim to show that $\mu^{*}(v) \geq \frac{20}{7}$.
Note that $G$ has no $1$-vertex by \autoref{Claim:pcf-n1}.
If $v$ is a $2$-vertex, then every neighbor of $v$ is a $3^{+}$-vertex by \autoref{Lem:semi-PCF}. By \ref{rule:c=5_1}, $\mu^{*}(v) = 2+ 2 \times \frac{3}{7} = \frac{20}{7}$.

If $v$ is a $3$-vertex, then $v$ sends charge only to $2$-neighbors.
If $v$ has no $2$-neighbors, then $v$ is not involved in the discharging process, resulting in $\mu^{*}(v) = 3 > \frac{20}{7}$.
On the other hand, if $v$ has a $2$-neighbor, then $n_3(v) = 0$ by \autoref{Claim:combine}, and \autoref{Lem:semi-PCF} further implies that $v$ has precisely one $2$-neighbor.
Since $v$ also has two $4^{+}$-neighbors, by the rules, $\mu^{*}(v) = 3 - \frac{3}{7} + 2 \times \frac{4}{21} > \frac{20}{7}$.

In the case where $v$ is a $4$-vertex, $v$ has at most three $2$-neighbors by \autoref{Lem:semi-PCF}.
If $v$ has at most one $2$-neighbor, then by the rules, $\mu^{*}(v) \geq 4 - \frac{3}{7} - 3 \times \frac{4}{21} > \frac{20}{7}$.
If $v$ has two or three $2$-neighbors, then \autoref{Claim:combine} implies that $v$ has neither a $3$-neighbor $u$ with $n_2(u)\geq 1$ nor a $4$-neighbor $w$ with $n_2(w) \geq 3$.
Hence, by the rules, $\mu^{*}(v) \geq 4 - \max\{2 \times \frac{3}{7}, 3 \times \frac{3}{7} - \frac{4}{21}\} > \frac{20}{7}$.

For a $5^{+}$-vertex $v$, by the rules, $\mu^{*}(v) \geq d(v) - d(v) \times \frac{3}{7} = \frac{4}{7}d(v) \geq \frac{20}{7}$.

Thus, $\mu^{*}(v) \geq \frac{20}{7}$ for every vertex $v$ in $G$.
Since the sum of initial charges is at most $\frac{20}{7}|V(G)|$, we obtain that $\mu^{*}(v) = \frac{20}{7}$ for every vertex $v$ in $G$.
By observing the final charge of each vertex in $G$, we conclude that each vertex in $G$ is either a 2-vertex or a $5$-vertex with five $2$-neighbors. Therefore, $G$ is a graph obtained from a $5$-regular multigraph $G_{0}$ by subdividing every edge exactly once.
Since $G_{0}$ is not the simple graph $K_{6}$, we know that $G_{0}$ has a proper $5$-coloring $\phi$ by Brooks' Theorem.
Let $\phi$ be a proper $5$-coloring of $G_{0}$, and consider $\phi$ as a coloring of the $5$-vertices of $G$.
We then color the $2$-vertices of $G$ one by one as follows: for each $2$-vertex $v$, we assign a color not in $\phi(N(v)) \cup \phi_{*}(N(v))$; this is possible since $|\phi(N(v)) \cup \phi_{*}(N(v))|\leq 4$.
Now, $\phi$ is a PCF $5$-coloring of $G$, which is a contradiction.

\subsection{Six or more colors}\label{sub:c6}

The discharging rules for this case are as follows:
\begin{enumerate}[label = \textbf{R\arabic*.}, ref = R\arabic*]
    \item\label{madrule:1}
    Each $4^{+}$-vertex sends charge $\frac{c-2}{c+2}$ to each of its $2$-neighbors.
    \item\label{madrule:2}
    Each $4^{+}$-vertex sends charge $\frac{c-6}{3(c+2)}$ to each of its $3$-neighbors.
\end{enumerate}

Consider a vertex $v$.
Our goal is to prove that $\mu^{*}(v) \geq \frac{4c}{c+2}$.
Note that $G$ has neither a $1$-vertex nor a $2$-thread by \autoref{Claim:pcf-n1} and \autoref{Claim:pcf-2thread}, respectively.
In addition, a $3$-vertex $v$ has no $3^{-}$-neighbor by \autoref{Claim:combine}.
Thus, a $3^{-}$-vertex has no $3^{-}$-neighbor.

If $v$ is a $2$-vertex, then it receives charge $\frac{c-2}{c+2}$ from each neighbor by~\ref{madrule:1} since $v$ has only $4^{+}$-neighbors.
Hence, $\mu^{*}(v) = 2 + 2 \times \frac{c-2}{c+2} = \frac{4c}{c+2}$.
If $v$ is a $3$-vertex, then it receives charge $\frac{c-6}{3(c+2)}$ from each neighbor by \ref{madrule:2} since $v$ has only $4^{+}$-neighbors.
Hence, $\mu^{*}(v) = 3+3\times\frac{c-6}{3(c+2)}=\frac{4c}{c+2}$.

In the following analysis, suppose that $v$ is a $4^{+}$-vertex.
If $c=6$, then $v$ sends charge only to its $2$-neighbors and none to its $3$-neighbors.
\autoref{Lem:semi-PCF} implies that $n_2(v) \leq 2d(v) -6$, so $\mu^{*}(v)\geq d(v)-(2d(v)-6)\times\frac{c-2}{c+2} = \frac{4c}{c+2}$.

Now suppose $c\geq 7$.
The discharging rules imply $\mu^{*}(v)\geq d(v)-n_2(v)\times \frac{c-2}{c+2}-n_3(v)\times\frac{c-6}{3(c+2)}$.
Since $d(v)=n_{4^{+}}(v)+n_3(v)+n_2(v)$, the following holds:
\begin{equation}\label{eq1}
\mu^{*}(v) \geq n_{4^{+}}(v)+ \frac{2c+12}{3(c+2)}\times n_3(v) +\frac{4}{c+2}\times n_2(v).
\end{equation}
Since $1 \geq \frac{2c+12}{3(c+2)} > \frac{4}{c+2}$,
\eqref{eq1} implies $\mu^{*}(v)\geq d(v)\times \frac{4}{c+2}$.
If either $d(v)\geq c$ or $n_2(v)+n_3(v)=0$, then $\mu^{*}(v)\geq \min\{4,\frac{4c}{c+2}\}=\frac{4c}{c+2}$.
Thus we may assume that $d(v)\leq c-1$ and $n_2(v)+n_3(v)\geq 1$. By \autoref{Lem:semi-PCF-2}, $n_2(v) + n_3(v) \leq 2d(v) -c$, so \eqref{eq1} implies
\[
\mu^{*}(v)\geq
(c-d(v)) +\frac{4}{c+2}\times (2d(v)-c) = \frac{(c - 6)(c - d(v)) + 4c}{c+2} > \frac{4c}{c+2},
\]
where the penultimate inequality holds since $c\geq 7$ and $d(v)\leq c-1$.

Thus $\mu^{*}(v) \geq \frac{4c}{c+2}$ for every vertex $v$ in $G$. Since the initial charge sum is at most $\frac{4c}{c+2}|V(G)|$, we obtain that $\mu^{*}(v) = \frac{4c}{c+2}$ for every vertex $v$ in $G$.

$\bullet$ $c = 6$. By examining the final charge of each vertex in $G$, each $3^{+}$-vertex in $G$ is actually a $d$-vertex with exactly $2d-6$ $2$-neighbors for $d \in \{3, 4, 5, 6\}$ (note that $2d - 6 \leq d$). Since a $6$-vertex has only $2$-neighbors, for $d \in \{3,4,5\}$, a $d$-vertex must have a $d'$-neighbor for some $d' \in \{3,4,5\}$, which is impossible by \autoref{Claim:combine}. Thus, every $3^{+}$-vertex of $G$ is a $6$-vertex with precisely six $2$-neighbors.

$\bullet$ $c \geq 7$. By examining the final charge of each vertex in $G$, each vertex in $G$ is either a $c$-vertex with exactly $c$ $2$-neighbors or a $3^{-}$-vertex. Since a $3^{-}$-vertex has no $3^{-}$-neighbor, no vertex in $G$ can be a neighbor of a $3$-vertex. Thus, every $3^{+}$-vertex of $G$ is a $c$-vertex with exactly $c$ $2$-neighbors.

In each case, $G$ is obtained from a $c$-regular multigraph $G_{0}$ by subdividing every edge exactly once. Since $G_{0}$ is not the simple graph $K_{c+1}$, it follows that $G_{0}$ is properly $c$-colorable by Brooks' Theorem. Let $\phi$ be a proper $c$-coloring of $G_{0}$, and consider $\phi$ as a coloring of the $c$-vertices of $G$. We then color the 2-vertices of $G$ one by one as follows: for each $2$-vertex $v$, we assign a color not in $\phi(N(v))\cup\phi_{*}(N(v))$; this is possible since $|\phi(N(v)) \cup \phi_{*}(N(v))| \leq 4$. Now, $\phi$ is a PCF $c$-coloring of $G$, which is a contradiction.

\section{Proof of \autoref{ODD-MAD-WY}}\label{Appendix:B}
\resetcounter

Observe that every graph in $\mathcal{G}_{c}$ has odd chromatic number at least $c+1$, so it suffices to prove the converse direction in \autoref{ODD-MAD-WY}. Fix $c\geq 5$, and let $G$ be a counterexample to \autoref{ODD-MAD-WY} with the minimum number of vertices.
So, $G$ is a graph with $\mad(G)\leq \frac{4c}{c+2}$ and $G \notin \mathcal{G}_{c}$ where $G$ has no odd $c$-coloring.

\begin{claim}\label{semi-odd}
For any vertex $x$, we always have that $(G, \{x\})$ has a semi-odd $c$-coloring.
\end{claim}

The proof of this claim is nearly identical to the proof of \autoref{Claim:semi-pcf} of \autoref{PCF-MAD-WY}.

\begin{claim}\label{n1}
$G$ has no $1$-vertices.
\end{claim}
\begin{proof}
Suppose to the contrary that $x$ is a $1$-vertex in $G$, and $u$ is the unique neighbor of $x$ in $G$. By \autoref{semi-odd}, $(G, \{x\})$ has a semi-odd $c$-coloring $\phi$. Therefore, we can assign a color not in $\{\phi(u), \phi_{o}(u)\}$ on $x$ to obtain an odd $c$-coloring of $G$, a contradiction.
\end{proof}

\begin{claim}\label{adj-two}
No two $2$-vertices in $G$ are adjacent.
\end{claim}
\begin{proof}
Suppose to the contrary that $x$ and $y$ are two adjacent $2$-vertices in $G$. By \autoref{semi-odd}, $(G, \{x\})$ has a semi-odd coloring, and the restriction of this coloring on $G - \{x\} \cup N_{2}(x)$ is a semi-odd coloring of $(G, \{x\} \cup N_{2}(x))$. However, by \autoref{Lem:semi-ODD}, we have that $n_{2}(x) + n_{e}(x) + c \leq 2d(x) = 4$, a contradiction.
\end{proof}

\begin{claim}\label{noneasy}
If $v$ is a $3$-vertex, then $5 \leq c \leq 6$ and $v$ is adjacent to at least two non-easy $4^{+}$-vertices. Moreover, if $n_2(v) + n_e(v) \geq 1$, then $c = 5$.
\end{claim}
\begin{proof}
Let $x, y, z$ be the three neighbors of $v$. By \autoref{semi-odd}, $(G, \{v\})$ has a semi-odd $c$-coloring. Thus, the restriction of this coloring on $G - (\{v\} \cup N_{2}(v))$ is a semi-odd $c$-coloring of $(G, \{v\} \cup N_{2}(v))$. It follows from \autoref{Lem:semi-ODD} that $n_2(v) + n_e(v) + c \leq 6$. Therefore, $5 \leq c \leq 6$, and $v$ has at least two non-easy $4^{+}$-neighbors. If $n_2(v) + n_e(v) \geq 1$, then $c \leq 6 - (n_2(v) + n_e(v)) \leq 5$.
\end{proof}

\begin{claim}\label{NofEasy}
Let $v$ be an easy $4^{+}$-vertex. Then $n_2(v) + n_{e}(v) \leq \min\{d(v), 2d(v)-c\}$. 
\end{claim}
\begin{proof}
According to \autoref{semi-odd}, $(G, \{v\})$ has a semi-odd $c$-coloring, and the restriction of this coloring to $G - (\{v\}\cup N_{2}(v))$ is a semi-odd $c$-coloring of $(G, \{v\}\cup N_{2}(v))$. By \autoref{Lem:semi-ODD}, we have $n_2(v) + n_{e}(v) \leq 2d(v)-c$. Observe that $n_2(v) + n_{e}(v) \leq d(v)$. Hence, $n_2(v) + n_{e}(v) \leq \min\{d(v), 2d(v)-c\}$. 
\end{proof}

Let $\epsilon$ be a sufficiently small positive real number. Let $\mu(v)$ be the initial charge of a vertex $v$, with $\mu(v)=d(v)$.
Let $\mu^{*}(v)$ be the final charge of $v$ after applying the following discharging rules:
\begin{enumerate}[label = \textbf{R\arabic*.}, ref = R\arabic*]
    \item\label{mad-rule:1}
    Each $2$-vertex receives charge $\frac{c-2}{c+2}$ from each $3^{+}$-neighbors
    \item
    Each $3$-vertex $v$ receives charge $\frac{2c-8+2\epsilon}{(c+2)t}$ from each non-easy $4^{+}$-neighbor, where $t$ is the number of non-easy $4^{+}$-neighbors of $v$.
    \item
    Each $4^{+}$-vertex receives charge $\frac{1+\epsilon}{c+2}$ from each non-easy $4^{+}$-neighbor.
\end{enumerate}

Note that the total charge is preserved during the discharging process.  Consider a vertex $v$.
We will show that $\mu^{*}(v) \geq \frac{4c}{c+2}$. Note that $\frac{c-2}{c+2} > \frac{c-4+\epsilon}{c+2} \geq \frac{1+\epsilon}{c+2}$.

Let $v$ be a 2-vertex. Then it receives charge $\frac{c-2}{c+2}$ from each of its $3^{+}$-neighbors according to \ref{mad-rule:1}. By \autoref{adj-two}, $v$ has two $3^{+}$-neighbors.
Hence, $\mu^{*}(v)=2+\frac{c-2}{c+2}\times 2=\frac{4c}{c+2}$.

Let $v$ be a $3$-vertex. By \autoref{noneasy}, $v$ has at least two non-easy $4^{+}$-neighbors. The total charge that $v$ receives from its non-easy $4^{+}$-neighbors remains the same no matter how many such neighbors it has. Thus, $\mu^{*}(v) \geq 3 + \frac{2c-8+2\epsilon}{c+2} - \frac{c-2}{c+2} = \frac{4c + 2\epsilon}{c+2} > \frac{4c}{c+2}$. 

Let $v$ be an easy $4^{+}$-vertex. By \autoref{NofEasy}, we have $n_2(v) + n_{e}(v) \leq \min\{d(v), 2d(v)-c\}$.
If $d(v)\geq c$, then
\[
\mu^{*}(v) \geq d(v)-\frac{c-2}{c+2} \times d(v) = \frac{4d(v)}{c+2} \geq \frac{4c}{c+2}.
\]
If $d(v) < c$, then $n_2(v) + n_{e}(v) \leq 2d(v)-c < d(v)$ and $v$ has at least one non-easy $4^{+}$-neighbor,
\[
\mu^{*}(v) \geq d(v)-\frac{c-2}{c+2}\times(2d(v)-c) + \frac{1 + \epsilon}{c+2} = \frac{1 + \epsilon - (c-6)d(v)+c(c-2)}{c+2} > \frac{4c}{c+2}.
\]
The last inequality can be easily verified in two cases: $c = 5$ and $c \geq 6$. 

Let $v$ be a non-easy $4^{+}$-vertex. Then it has no $2$-neighbors. If $v$ is not adjacent to a $3$-vertex, then 
\[
\mu^{*}(v) \geq d(v) - \frac{1+\epsilon}{c+2} \times d(v) \geq \frac{4(c+1-\epsilon)}{c+2} > \frac{4c}{c+2}.
\]
Assume $v$ has a $3$-neighbor. By \autoref{noneasy}, we know that $5 \leq c \leq 6$. Suppose that $v$ sends charge $\frac{c-4+\epsilon}{c+2}$ to a $3$-neighbor $u$. By the discharging rules, $u$ is a $3$-vertex with exactly two non-easy $4^{+}$-neighbors, thus $n_{2}(u) + n_{e}(u) = 1$. By \autoref{noneasy}, we have $c = 5$. Then 
\[
\mu^{*}(v) \geq d(v) - \frac{c-4+\epsilon}{c+2} \times d(v) \geq \frac{4(6-\epsilon)}{c+2} > \frac{4c}{c+2}.
\]
If $v$ does not send charge $\frac{c-4+\epsilon}{c+2}$ to any $3$-neighbor, then 
\[
\mu^{*}(v) \geq d(v) - \max\left\{\frac{2c-8+2\epsilon}{3(c+2)}, \frac{1+\epsilon}{c+2}\right\} \times d(v)> \frac{4c}{c+2} \text{ (note that $5 \leq c \leq 6$)}.
\]

Since $\mad(G) \leq \frac{4c}{c+2}$, the initial charge sum is at most $\frac{4c}{c+2}|V(G)|$. By the above argument, the final charge sum is at least $\frac{4c}{c+2}|V(G)|$. Thus, we conclude that $\mu^{*}(v)=\frac{4c}{c+2}$ for every vertex $v$.

From the discharging procedure, we deduce that $G$ has only $2$-vertices and $c$-vertices, and every $c$-vertex has only $2$-neighbors. Thus, $G$ is obtained from a $c$-regular (multi)graph $G_{0}$ by subdividing every edge of $G_{0}$. Since $G_{0}$ is not the simple graph $K_{c+1}$, we know that $G_{0}$ has a proper $c$-coloring $\phi$ by Brooks' Theorem. Consider $\phi$ as a coloring of the $c$-vertices of $G$, and then we proceed to color the 2-vertices of $G$ one by one as follows: for each $2$-vertex $v$, color it with a color not in $\phi(N(v))\cup\phi_{o}(N(v))$. Now, $\phi$ is an odd $c$-coloring of $G$, which contradicts our assumption that $G$ had no odd $c$-coloring.
\end{document}